\documentclass{article}
\usepackage{graphicx} 
\usepackage[T1]{fontenc}
\usepackage[utf8]{inputenc}
\usepackage[english]{babel}
\usepackage{amsfonts}
\usepackage{amssymb}
\usepackage{amsmath}
\usepackage{amsthm}
\usepackage{thmtools}
\usepackage[all]{xy}
\usepackage{geometry}
\usepackage{tikz-cd}
\usepackage{stmaryrd}
\usepackage[nottoc, notlof, notlot]{tocbibind}
\usepackage[pdftex]{hyperref} 
\usepackage{cleveref}
\usepackage{MnSymbol}
\usepackage{mathtools}
\usepackage{bbm}
\usepackage{dsfont}
\usepackage{blindtext}
\usepackage{titlesec}
\usepackage{shuffle}
\usepackage[
  backend=biber,
  style=alphabetic
]{biblatex}

\usepackage{combelow}
\usepackage{enumitem}
\setenumerate[1]{label=(\arabic*)}
\usepackage{csquotes}

\DeclareMathOperator{\Spec}{Spec}

\DeclareMathOperator{\Ext}{Ext}

\DeclareMathOperator{\Isom}{Isom}

\DeclareMathOperator{\dR}{dR}

\DeclareMathOperator{\di}{d}

\DeclareMathOperator{\B}{B}

\DeclareMathOperator{\Gr}{Gr}

\DeclareMathOperator{\dch}{dch}
\DeclareMathOperator{\Aut}{Aut}

\theoremstyle{definition}
\newtheorem{defi}{Definition}[section]
\newtheorem{rem}[defi]{Remark}
\newtheorem{ex}[defi]{Example}
\newtheorem{nota}[defi]{Notation}

\theoremstyle{plain}
\newtheorem{lemma}[defi]{Lemma}
\newtheorem{thm}[defi]{Theorem}
\newtheorem{prop}[defi]{Proposition}
\newtheorem{conj}[defi]{Conjecture}
\newtheorem{cor}[defi]{Corollary}

\newcommand{\Z}{\mathbb{Z}}

\newcommand{\N}{\mathbb{N}}
\newcommand{\Q}{\mathbb Q}

\newcommand{\m}{\mathfrak{m}}

\newcommand{\MT}{\mathcal{MT}}
\newcommand{\A}{\mathcal A}

\title{The motivic Galois group of a double zeta value}
\author{Kenza Memlouk }
\date{}
\begin{document}

\maketitle
\begin{abstract}
We consider multiple zeta values, which are periods of mixed Tate motives over $\Z$. For a given multiple zeta value $\zeta$, there exists a unique minimal motive $M(\zeta)$ such that $\zeta$ is a period of $M(\zeta)$. In general, the motive $M(\zeta)$ is difficult to compute. In this article, we compute the minimal motive $M(a,b)$ associated to a given double zeta value $\zeta(a,b)$. We also compute the motivic Galois group $G(a,b)$ associated to $\zeta(a,b)$ and discuss its dimension. Moreover, we give a period matrix of $M(a,b)$. The period conjecture predicts that the dimension of $G(a,b)$ equals the transcendence degree of the algebra of periods of $M(a,b)$. Hence our results lead to conjectures about algebraic relations between single and double zeta values.

\end{abstract}
\tableofcontents

\section{Introduction}
This article is about double zeta values, which are special examples of periods, and their minimal motive. Let us briefly recall the general setting. Following the point of view of Kontsevich and Zagier \cite{kz}, a period is a complex number whose real and imaginary parts are integrals of rational functions on a semi-algebraic set, that is a set defined by inequalities of polynomials with coefficients in $\Q$. For example the number $\pi$ can be written as the following integral $$\pi=\int_{x^2+y^2\leqslant 1}\mathrm{d} x\mathrm{d} y$$ and thus $\pi$ is a period.\\

Grothendieck introduced an equivalent definition of periods as the $\Q$-vector space spanned by coefficients of the comparison isomorphism $\mathrm{comp}_{\dR,\B}$ induced by integration between de Rham and Betti cohomologies for varieties over $\Q$. Let us denote this $\Q$-vector space by $\mathcal P$. It is a countable subalgebra of $\mathbb C$. For an introduction to periods, see \cite{Fre19}.\\

Let us briefly explain how to construct the algebra of motivic periods. We give ourselves a category of motives to study periods, for example the category of Nori motives or the category of mixed Tate motives over $\Z$. This category of motives is endowed with the realisation functors $\omega_{\dR}$ and $\omega_{\B}$, corresponding respectively to de Rham and Betti cohomologies.  Consider $\underline{\mathrm{Isom}}^\otimes(\omega_{\dR},\omega_{\B})$ given by the Tannakian formalism. A special $\mathbb C$-point $\mathrm{comp}_{\dR,\B}$ is given by integration. Define the  algebra of motivic periods as $\mathcal P^{\mathfrak{m}}:=\mathcal O(\underline{\Isom}^\otimes(\omega_{\dR},\omega_{\B}))$. Then, evaluation at $\mathrm{comp}_{\dR,\B}$ induces a surjective ring homomorphism 
\begin{equation}\label{per1}
\mathrm{per}:\mathcal P^{\mathfrak{m}}\twoheadrightarrow \mathcal P
\end{equation} called the period map.\\

The period conjecture predicts that $\mathrm{per}$ is injective \cite{andre}. The algebra $\mathcal P^{\mathfrak{m}}$ is a torsor under the group $G_{\dR}:=\underline{\Aut}^\otimes(\omega_{\dR})$ of automorphisms of the de Rham realisation $\omega_{\dR}$ respecting the tensor product. This setting allows us to consider a Galois theory for (motivic) periods. This Tannakian point of view for periods was first introduced by Grothendieck and it was then developed by André \cite{andrebook}, Kontsevich, Nori. More recently, Ayoub \cite{ayoub}, Huber and Müller-Stach \cite{HMS17} have also developed this point of view. For an introduction to Galois theory of periods, see \cite{brownnotesonmotper}, \cite{galtheoryper}.\\

Given an element $\zeta$ in $\mathcal P^{\mathfrak{m}}$, there exists a unique minimal motive $M(\zeta)$ such that $\zeta$ is a motivic period of $M(\zeta)$ up to isomorphism. In other words, for all motives $M$ such that $\zeta$ is a motivic period of $M$, the motive $M(\zeta)$ is a sub-quotient of $M$. However, in general the motive $M(\zeta)$ and its Galois group are difficult to compute. \\

We are interested in multiple zeta values, which are defined as follows: if $\textbf{n}=(n_1,\dots,n_r)\in\N_{>0}^r$ is a multi-index, we define the \emph{multiple zeta value} associated to $\textbf{n}$ as the series $$\zeta(\textbf{n})=\sum_{k_r>k_{r-1}>\dots>k_1\geqslant 1}\frac 1 {k_1^{n_1}\dots k_r^{n_r}}\in \mathbb R$$ (converging for $n_r$ at least $2$). The quantity $r$ is called the \emph{depth} of $\zeta(\textbf{n})$. These numbers can be written as iterated integrals and thus they are periods of algebraic varieties. Multiple zeta values naturally appear in various contexts such as particle physics with Feynman integrals. For a survey on these numbers, see \cite{dup19}. \\ 

In the case of single zeta values, the minimal motive $M(n)$ such that $\zeta(n)$ is a period of $M(n)$ is known. For example, for odd $n\geqslant 3$, a period matrix of $M(n)$ is the following: $$P(n)=\begin{pmatrix}1&\zeta(n)\\
0&(2\pi i)^n
\end{pmatrix}.$$

Moreover, for odd $n\geqslant 3$, the Galois group of $M(n)$ is $$G(n)=\left\{\begin{pmatrix}
1&*\\
0&t^n
\end{pmatrix}
\right\} \subset GL_{2,\Q}$$ and $G(n)$ acts on $P(n)$ by right multiplication.\\
For an even positive integer $n$, a period matrix of $M(n)$ is given by the matrix with one coefficient $$P(n)=\begin{pmatrix}
\zeta(n)
\end{pmatrix}$$
where $\zeta(n)$ is a rational multiple of $\pi^n$. The motivic Galois group of $M(n)$ is $\mathbb G_m$.\\
In particular, the dimension of the group is at most $2$ for all integers $n\geqslant 2$. 
We will see that the result is more complicated for the minimal motive $M(a,b)$ of a double zeta value $\zeta(a,b)$. For example, there is no finite upper bound for the dimension of the Galois group $G(a,b)$.\\

Let us recall the motivic techniques that we are going to use. Multiple zeta values are exactly the (real) periods of mixed Tate motives over $\Z$ \cite{Bro12a}\cite{DG05}\cite{Ter01}. Let us denote this category by $\MT(\Z)$. This Tannakian category was constructed by Levine \cite{levineart}. Let $G_{\dR}$ be its Tannakian group induced by the de Rham realisation functor. The group $G_{\dR}$ can be written as a semi-direct product $$U_{\dR}\rtimes \mathbb G_m$$ where $U_{\dR}$ is prounipotent. Hence, the datum of mixed Tate motive over $\Z$ is exactly the datum of a graded representation of $U_{\dR}$.\\

Let us denote by $\widetilde{\mathcal Z}^\m:=\mathcal O(\underline{\Isom}^\otimes((\omega_{\dR})_{\vert\MT(\Z)},(\omega_{\B})_{\vert\MT(\Z)}))\subset \mathcal P^\m$ the algebra of motivic periods of $\MT(\Z)$. It is constructed in the same way as $\mathcal P^\m$ but in the specific category of mixed Tate motives over $\Z$. In \cite{Bro12a}, Brown constructs the $G_{\dR}$-module of motivic multiple zeta values $\mathcal H^\MT\subset\widetilde{\mathcal Z}^{\mathfrak{m}}.$  There are canonical elements $\zeta^\m(n_1,\dots,n_r)$ in $\mathcal H^\MT$ and a surjective period map $$\mathrm{per}:\left\{\begin{array}{lll}
\mathcal H^\MT&\twoheadrightarrow & \mathcal Z\\
\zeta^\m(n_1,\dots,n_r)&\mapsto&\zeta(n_1,\dots,n_r)
\end{array}
\right.$$ where $\mathcal Z$ denotes the $\Q$-vector space generated by multiple zeta values. This period map is the restriction of the map per in \eqref{per1} to $\mathcal H^\MT$.\\

For a general motivic multiple zeta value, it can be difficult to compute its minimal motive and the associated Tannakian group. The goal of this paper is to compute the minimal motive $M(a,b)$ associated with a motivic double zeta value $\zeta^\m(a,b)$. In particular, we compute the Tannakian group $G(a,b)$ of $\zeta^\m(a,b)$.

\begin{thm}(For a more precise statement, see \Cref{grosthmtangp}). Let $a$, $b$ be positive integers. Then the group $G(a,b)$ is computed and its expression depends on $a$ and $b$. In particular, we deduce the dimension of $G(a,b)$ in \Cref{dimi}.
\end{thm}

\begin{ex}
Let us give a period matrix $P(3,7)$ for the minimal motive for $\zeta(3,7)$. We have
$$P(3,7)=\begin{pmatrix}
        1&\zeta(5)&\zeta(7)&\zeta(3,7)\\
        0&(2\pi i)^5&0&(2\pi i)^5\zeta(5)\\
        0&0&(2\pi i)^7&(2\pi i )^7\zeta(3)\\
        0&0&0&(2\pi i)^{10}
    \end{pmatrix}\subset GL_{4,\Q}$$ 
    The motivic Galois group $G(3,7)$ associated to $\zeta(3,7)$ is given by $$G(3,7)=\left\{\begin{pmatrix}
        1&\alpha_5&\alpha_7&\alpha_{10}\\
        0&t^5&0&-3\alpha_5\\
        0&0&t^7&\alpha_3\\
        0&0&0&t^{10}
    \end{pmatrix}\right\}_{\alpha_i\in\Q}\subset GL_{4,\Q}$$ and $G(3,7)$ acts on $P(3,7)$ by right multiplication.
\end{ex}
 Moreover, we compute a period matrix of $M(a,b)$ in \Cref{permat}. Let us recall that the period conjecture predicts that the dimension of $G(a,b)$ is equal to the transcendence degree of the algebra generated by periods of $M(a,b)$. Hence our results give a prediction about algebraic relations between single and double zeta values. \\

The starting point for the computation of $M(a,b)$ is to see that via the Tannakian formalism, the motive $M(a,b)$ is given by the $\Q$-vector space spanned by the orbit of the motivic multiple zeta value $\zeta^\m(a,b)$ under the action of $G_{\dR}$, when viewed as representation of $G_{\dR}$ \cite{brownnotesonmotper}.\\
It is not easy to compute the orbit of a given motivic multiple zeta value. However, the $G_{\dR}$-module $\mathcal H^\MT$ is isomorphic to the free coalgebra $\mathcal U$ cogenerated by some $f_i$'s with an explicit action of $G_{\dR}$. The benefit of working with the coalgebra $\mathcal U$ is that the orbit of a given word in the $f_i$'s is easy to compute. Hence, to compute the orbit of a given motivic multiple zeta value $\zeta^\m(\textbf{n})$, we have to decompose $\zeta^\m(\textbf{n})$ as a word in the $f_i$'s for odd $i\geqslant 3$. Brown  constructs a method to do such a decomposition in \cite{Bro12b}. At each step, we compute the coaction on motivic iterated integrals and it reduces the depth by one. In this paper, we apply this method to double zeta values to obtain explicit formulas. To do so, we use a derivation operator introduced by Brown in \cite{Bro12a}. This operator can be explicitly computed thanks to the action of $G_{\dR}$ on $\mathcal H^\MT$. \\

Even though Brown explains a method to do such a decomposition, it is not easy to compute it in practice. We do the computations only for double zeta values because of some technical difficulties. More precisely, for $(a+b)$ being odd, when writing $\zeta^\m(a,b)$ as a linear combination of words in the cogenerators $f_i$'s, the coefficient in front of $f_{a+b}$ cannot be computed by Tannakian methods. Hence, we apply the period map to the linear combination of words in the $f_i$'s and we use relations between real double zeta values from \cite{Borwein} to find the remaining coefficient. Since these relations are valid only for double zeta values, our method cannot be directly adapted to multiple zeta values with bigger depth.

\subsection*{Organisation of the paper}
We first review some facts about multiple zeta values. We give their definition as a series and we explain how they can be obtained as iterated integrals in Section \ref{subsectiterint}. Then, we introduce the category $\MT(\Z)$ of mixed Tate motives over $\mathbb Z$ in Section \ref{subsectmtm}. We give some properties of its Tannakian group.\\
After that, we present motivic multiple zeta values (Section \ref{subsectmotiterint}). We recall some derivation operators, which were introduced by Brown in \Cite{Bro12a} (Section \ref{subsectcoact}). These derivation operators are useful to compute the decomposition of a multiple zeta value as a word in the cogenerators $f_i$. \\
Finally, we apply the derivation operators to double zeta values (Section \ref{subsectcoactdouble}) and use these computations to decompose them as a word in the cogenerators $f_i$'s (Section  \ref{subsectwritingfalpha}).\\
The decomposition of a given motivic double zeta value $\zeta^\m(a,b)$ makes it possible to compute its orbit under the action of $G_{\dR}$. Doing so, we can exhibit the minimal motive $M(a,b)$ associated to $\zeta^\m(a,b)$. Then, we compute the Tannakian group $G(a,b)$ of $M(a,b)$ (Section \ref{subsecttangroup}). 
We also give the dimension of $G(a,b)$ and compute a period matrix of $M(a,b)$ (Section \ref{transsection}). This allows us to make some conjectures about algebraic relations between single and double zeta values.

\subsection*{Conventions}\label{Conv}
\begin{enumerate}
\item We denote by $\N$ the set of positive integral numbers (without $0$).
\item We denote by $\N_0$ the set of non-negative integral numbers (including $0$).
\item We denote by $\mathfrak S (n)$ the group of permutations of $\{ 1,\dots,n\}$ for $n$ a positive integer.
\item\label{shufflesymbol} If $r,s$ are positive integers, we denote the set $\mathfrak S (r,s)$ of $(r,s)$-shuffles which is defined as follows $$\mathfrak S(r,s)=\{\sigma\in\mathfrak S(r+s)\mid \sigma(1)<\sigma(2)<\dots<\sigma(r)\;\mathrm{and}\;\sigma(r+1)<\sigma(r+2)<\dots<\sigma(r+s)\}.$$
\item We denote by $\MT(\Z)$ the category of mixed Tate motives over $\Z$.
\item If $a$ and $b$ are positive integers, we denote by $M(a,b)$ the minimal motive in $\MT(\Z)$ such that $\zeta^\m(a,b)$ is a period of $M(a,b)$.
\item We denote by $G(a,b)$ the Tannakian group of $M(a,b)$.
\item In this paper, we consider the motivic weight in $\MT(\Z)$ which is half the weight in the category of Hodge structures.
\end{enumerate}
\subsection*{Acknowledgments}  
I would like to thank C. Dupont in particular for suggesting this problem to me and for his invaluable insights and comments on earlier versions. I would also like to thank G. Ancona, D. Fr\u{a}\cb{t}il\u{a}, J. Fres\'an and  A. Huber for many helpful explanations and remarks.\\
This research was partly supported by the grant ANR–23–CE40–0011 of Agence Nationale de la Recherche.

\section{Iterated integrals}\label{subsectiterint}
In this section, we recall generalities on multiple zeta values. First, we give their expression as series (Definition \ref{mzvdef}). They provide a way to generalize the Riemann zeta function by allowing the argument to be a tuple. Then, we describe iterated integrals, through which one can represent multiple zeta values as periods in Kontsevich and Zagier's sense. The notion of iterated integrals was introduced by Chen in \cite{Che77}. To express multiple zeta values as iterated integrals, we introduce a piece of notation associating a sequence of zeros and ones to a tuple (\Cref{01seqce}). Then, we define the shuffle-regularized multiple zeta values (Definition \ref{regumzv}). For an introduction to the theory of multiple zeta values, see the book \cite{BGF17}. \\

\begin{defi}\label{admituple}
    A multi-index $\textbf{n}=(n_1,\dots,n_r)$ in $\mathbb Z^r$ is called \emph{positive} if for all $1\leqslant i\leqslant r$, the integer $n_i$ is positive. A positive multi-index is called \emph{admissible} if $n_r\geqslant 2$. The \emph{weight} of $\textbf{n}$ is $n_1+\dots+n_r$ and its \emph{depth} is $r$.
\end{defi}

\begin{defi}\label{mzvdef}
    Let $\textbf{n}=(n_1,\dots,n_r)\in\mathbb Z^r$ be an admissible multi-index. Let us define the real number: $$\zeta(\textbf{n})=\sum_{k_r>k_{r-1}>\dots>k_1\geqslant 1}\frac 1 {k_1^{n_1}\dots k_r^{n_r}}$$ which is named the \emph{multiple zeta value} associated to $\textbf{n}$. The \emph{depth} of the multiple zeta value $\zeta(n_1,\dots, n_r)$ is $r$. The \emph{weight} of $\zeta(n_1,\dots,n_r)$ is $n_1+\dots+n_r$.
\end{defi}
\begin{rem}
     If $\textbf{n}$ is admissible, the series converges and $\zeta(\textbf{n})$ is a real number.
\end{rem}

\begin{rem} 
    It is possible to find another convention in the literature, where the indices are in the inverse order, for example in \cite{Borwein}. In that other convention a positive multi-index $\textbf{n}=(n_1,\dots,n_r)\in\N^r$ is admissible if $n_1\geqslant 2$. Let us write an exponent $op$ when we work with the other convention. We define $$\zeta^{op}(\textbf{n})=\sum_{k_1>k_{2}>\dots>k_r\geqslant 1}\frac 1 {k_1^{n_1}\dots k_r^{n_r}}$$ which is $\zeta(n_r,\dots,n_1)$ with our convention. We stick to the convention of \Cref{mzvdef} in this paper.
\end{rem}
\begin{prop}\label{doubleshuffle}
Let $N$ be a positive integer. We have the \emph{double-shuffle relations}: 
$$\zeta(j,N-j)+\zeta(N-j,j)+\zeta(N)=\sum_{k=2}^{N-1}\left(\binom{k-1}{j-1}+\binom{k-1}{N-j-1}\right)\zeta(N-k,k)$$ for any positive integer $2\leqslant j\leqslant N-2$.

\end{prop}
\begin{defi}\label{dch}
    Let $\dch$ be the path defined as follows: $$\dch:\left\{\begin{array}{clc} [0,1]&\rightarrow &\mathbb C\\
t&\mapsto &t
\end{array}
\right.$$ which we call the \emph{straight path}.
\end{defi}

\begin{defi}[Iterated Integrals]\label{itintegr} 
Let $\epsilon_0,\dots,\epsilon_{n+1}\in\{0,1\}$. When it converges, we define: $$I_{\dch}(\epsilon_0;\epsilon_1 \dots \epsilon_n;\epsilon_{n+1})=\int_{\dch} \frac{\di t_1}{\epsilon_1-t_1}\wedge\dots\wedge\frac{\di t_n}{\epsilon_n-t_n}$$ which we call \emph{iterated integral}.
\end{defi}

\begin{defi}\label{01seqce}
    Let $(n_1,\dots,n_r)$ be a tuple in $\N^r$. Let us define $$\varepsilon(n_1,\dots,n_r)=\underbrace{10\dots 0}_{n_1}\underbrace{10\dots 0}_{n_2}\dots\underbrace{10\dots 0}_{n_r}$$ which is in $\{0,1\}^{n_1+\dots+n_r}$.
\end{defi}
\begin{defi}\label{admi01}
   Let $(\epsilon_1,\dots,\epsilon_{N})\in\{0,1\}^{N}$. 
    We say that this tuple is \emph{admissible} if $\epsilon_1=1$ and $\epsilon_N=0$.
\end{defi}
\begin{rem}
We keep notations of \Cref{admi01}. The tuple $(\epsilon_1,\dots,\epsilon_N)$ is admissible if and only if it is equal to $\varepsilon(n_1,\dots,n_r)$ for some admissible tuple $(n_1,\dots,n_r)$ in the sense of \Cref{admituple}. 
\end{rem}

\begin{prop}\label{mzvetintegr}
Let $(n_1,\dots,n_r)$ be admissible (\Cref{admituple}). Let $\dch$ be the straight path defined in \Cref{dch}. The iterated integral $I_{\dch}(0;\varepsilon(n_1,\dots,n_r);1)$ is well-defined and the equality $$\zeta(n_1,\dots,n_r)=(-1)^rI_{\dch}(0;\varepsilon(n_1,\dots,n_r);1)$$ holds.
\end{prop}

\begin{lemma} \label{lemitintegr}

Let $\dch$ be as in \Cref{dch}. 
    There is a unique way to define real numbers $I(\epsilon_0;\epsilon_1\dots \epsilon_n;\epsilon_{n+1})$ such that:
    \begin{enumerate}
    \item \label{coincide}they coincide with $I_{\dch}(0;\epsilon_1\dots \epsilon_n;1)$ given by \Cref{itintegr} if $\epsilon_0$ is $0$, $\epsilon_{n+1}$ is $1$ and $(\epsilon_1\dots \epsilon_n)$ is admissible;
    \item the equalities $I(\epsilon_0;\epsilon_1;\epsilon_2)=0$ and $I(\epsilon_0;\epsilon_1)=1$ hold for all $\epsilon_0,\epsilon_1,\epsilon_2\in\{0,1\}$;
    \item we have the shuffle product formula $$\forall \epsilon_i,x,y\in\{0,1\}, I(x;\epsilon_1\dots \epsilon_r;y)I(x;\epsilon_{r+1}\dots \epsilon_{r+s};y)=\sum_{\sigma\in\mathfrak S(r,s)}I(x;\epsilon_{\sigma(1)}\dots \epsilon_{\sigma(r+s)};y)$$ where $\mathfrak S(r,s)$ is the set of $(r,s)$-shuffles introduced in \Cref{Conv}\ref{shufflesymbol};
    \item the equality $I(\epsilon_0;\epsilon_1\dots \epsilon_n;\epsilon_{n+1})=0$ holds if $\epsilon_0=\epsilon_{n+1}$ and $n\geqslant 1$;
    \item the equality $I(\epsilon_0;\epsilon_1\dots \epsilon_n;\epsilon_{n+1})=(-1)^nI(\epsilon_{n+1};\epsilon_n\dots \epsilon_1;\epsilon_0)$ holds;
 \item the equality $I(\epsilon_0;\epsilon_1\dots \epsilon_n;\epsilon_{n+1})=I(1-\epsilon_0;(1-\epsilon_n)\dots (1-\epsilon_1);1-\epsilon_{n+1})$ holds.
\end{enumerate}
    \end{lemma}
    \begin{defi}[Regularization of multiple zeta values]\label{regumzv}
        We define \emph{shuffle-regularized multiple zeta values} as follows $$\zeta(n_1,\dots, n_r)=(-1)^rI(0;n_1\dots n_r;1)$$ for all tuples of integers $(n_1,\dots,n_r)$.
    \end{defi}
    \begin{rem}
        For an admissible tuple, this definition coincides with multiple zeta values by \Cref{mzvetintegr}.
    \end{rem}
\section{Tannakian formalism and mixed Tate motives over $\mathbb Z$ }\label{subsectmtm}
In this section, we recall a few facts about the category $\MT(\Z)$ of mixed Tate motives over $\Z$. Then, we describe its Tannakian structure in \Cref{grosthmpreli}\ref{notamotmzv}. 

\begin{nota}
    Let $\MT(\mathbb Z)$ denote the category of mixed Tate motives over $\mathbb Z$ from \cite{levineart}(see also  \cite{BGF17}).
    \end{nota}
    \begin{rem}
As stated in \cite{DG05}, multiple zeta values are periods of motives in $\MT(\Z)$. The converse is true, id est a period of a motive in $\MT(\Z)$ is a $\Q[2\pi i]$-linear combination of multiple zeta values as proven in \cite{Bro12a}.
\end{rem}
    \begin{rem}
        The category $\MT(\Z)$ is a Tannakian category whose simple objects are Tate twists $\mathbb Q(n)$. Each object of $\MT(\Z)$ is endowed with a weight filtration. For example the motive $\Q(n)$ has weight $-n$ for any integer $n$ \cite{levineart}.
    \end{rem}
    \begin{nota}\label{tangroup}
    We denote by $G_{\dR}$ the Tannakian group associated with $\MT(\Z)$ and its de Rham realisation functor $\omega_{\dR}$.
    \end{nota}
     \begin{thm}\label{grosthmpreli}
     \begin{enumerate}
     \item \label{extmtm} The $\Ext^1$ groups are determined by: $$\Ext^1_{\MT(\mathbb Z)}(\mathbb Q(i),\mathbb Q(j))=\left\{\begin{array}{lll} \mathbb Q&\mathrm{ if}\;j-i\geqslant 3\;\mathrm{and}\;j-i\;\mathrm{odd}\\
    0&\mathrm{otherwise}
    \end{array}
    \right.$$ and the $\Ext^k$ groups vanish for $k\geqslant 2$.
    \item \label{notamotmzv} The group scheme $G_{\dR}$ introduced in Notation \ref{tangroup} can be decomposed as the following semi-direct product $$G_{\dR}=U_{\dR}\rtimes\mathbb G_m$$ where $U_{\dR}$ is the prounipotent group over $\mathbb Q$ whose graded Lie algebra $\mathfrak u^{gr}$ is non-canonically isomorphic to the free Lie algebra $\mathbb L(\sigma_3,\sigma_5,\dots,\sigma_{2n+1},\dots)$ with generators $\sigma_{2n+1}$ in degree $-(2n+1)$ for all positive integer $n$.
  
    \end{enumerate}
    \end{thm}

\section{Motivic iterated integrals}\label{subsectmotiterint}
The goal of this section is to recall a few facts about motivic multiple zeta values. To do so, we introduce a cofree Hopf algebra with cogenerators $f_i$ which are going to be a key ingredient in the computation of the minimal motive of a double zeta value. Then, we define motivic iterated integrals in \Cref{motmzvdef}. They are a motivic lift of the iterated integrals of \Cref{itintegr}. 
\begin{defi}\label{amt} We use notations of \Cref{grosthmpreli}\ref{notamotmzv}. 
    Let us denote $\A^{\MT}$ the graded commutative Hopf $\Q$-algebra of affine functions on $U_{\dR}$ with coproduct denoted by $\Delta$. 
    \end{defi}
    \begin{defi}\label{shuffle}
        Let us define $\mathcal U'$ to be the cofree Hopf algebra $$\mathcal U'=\Q\langle f_3,f_5,\dots,f_{2r+1},\ldots\rangle$$ whose cogenerators are the $f_{2r+1}$'s in degree $2r+1$ for every positive integer $r$. It is a Hopf algebra with coproduct $\Delta$ given by deconcatenation: $$\Delta:\left\{\begin{array}{cll}
        \mathcal U'&\rightarrow &\mathcal U'\otimes_\Q\mathcal U'\\
        f_{i_1}\dots f_{i_r}&\mapsto &1\otimes f_{i_1}\dots f_{i_r}+f_{i_1}\dots f_{i_r}\otimes 1+\sum_{k=1}^{r-1} f_{i_1}\dots f_{i_k}\otimes f_{i_{k+1}}\dots f_{i_r} 
        \end{array}
        \right.$$ and the multiplication given by the shuffle product $\shuffle$ such that $$f_{i_1}\dots f_{i_r}\shuffle f_{i_{r+1}}\dots f_{i_{s}}=\sum_{\sigma\in\mathfrak S(r,s)}f_{i_{\sigma(1)}}\dots f_{i_{\sigma(r+s)}}$$ where $\mathfrak S(r,s)$ is the set of $(r,s)$-shuffles introduced in \Cref{Conv}\ref{shufflesymbol}.
    \end{defi}
    \begin{rem} We keep notations of \Cref{shuffle}.
        As a $\Q$-vector space, the basis of $\mathcal U'$ is composed with all non-commutative words in the cogenerators $f_{i}$'s for $i\geqslant 3$ odd.
    \end{rem}
    \begin{prop}\cite[Subsection 2.5]{Bro12a}
    The Hopf algebra $\A^{\MT}$ of \Cref{amt} is non-canonically isomorphic to the cofree Hopf algebra $\mathcal U'$.
    \end{prop}
\begin{rem}\label{coord}
The words in the $f_i$'s are functions on $U_{\dR}$. In fact, the algebra $\mathcal U'$ is isomorphic to the polynomial algebra in the Lyndon words in the $f_i$'s \cite[Theorem 6.3.4]{hopfalg}. For example, the function $f_i$ is a coordinate function. The function $f_if_j$ for $i<j$ is also a coordinate function. However, for $i<j$, the function $f_jf_i=f_i\shuffle f_j-f_if_j$ can be written as a polynomial in Lyndon words but it is not itself a Lyndon word.
\end{rem}

\begin{defi}\label{comoda}
    We define the $\mathcal A^{\MT}$-comodule $\mathcal A^\MT\otimes_\Q\Q[f_2]$ with $f_2$ being of degree $2$. The coaction $\Delta$ is given by restriction and by $\Delta f_2=f_2\otimes 1$.
\end{defi}
\begin{thm}[Motivic iterated integrals]\cite[Subsections 3.2, 3.5]{Bro12b} \label{motmzvdef} There exists a graded $\mathbb Q$-algebra $$\mathcal H^\MT=\bigoplus_{n\in\N_0}\mathcal H_n\subset\mathcal O(G_{\dR})$$ of motivic multiple zeta values with the following properties.\\
\begin{enumerate}
    \item The underlying vector space of $\mathcal H_n$ is spanned by symbols $I^\m(\epsilon_0;\epsilon_1\dots \epsilon_n;\epsilon_{n+1})$ where the $\epsilon_i$'s are in $\{0,1\}$;
    \item The underlying graded vector space of $\mathcal H^\MT$ is a graded $\mathcal A^{\MT}$-comodule;
    \item \label{isoheta} We have an isomorphism of graded $\mathcal
    A^{\MT}$-comodule $$\mathcal H^\MT\cong \mathcal A^{\MT}\otimes_\Q\Q[f_2]$$ which sends $-I^\m(0;10;1)$ to $f_2$. The comodule $\A^\MT\otimes_\Q\Q[f_2]$ is introduced in \Cref{comoda}.
    \end{enumerate}
    Moreover, the symbols $I^\m(\epsilon_0;\epsilon_1\dots \epsilon_n;\epsilon_{n+1})$ verify the following relations.
\begin{enumerate}
\item\label{i0} \textbf{I0}: $I^{\mathfrak{m}}(\epsilon_0;\epsilon_1\dots \epsilon_n;\epsilon_{n+1})=0$ if $\epsilon_0=\epsilon_{n+1}$ and $n\geqslant 1$;
\item \textbf{I1}: $I^{\mathfrak m}(\epsilon_0;\epsilon_1;\epsilon_2)=0$ and $I^{\mathfrak m}(\epsilon_0;\epsilon_1)=1$ for all $\epsilon_0,\epsilon_1,\epsilon_2\in\{0,1\}$;
\item \label{i2} \textbf{I2}: $I^{\mathfrak m}(0;\epsilon_1\dots \epsilon_n;1)=(-1)^nI^{\mathfrak m}(1;\epsilon_n \dots \epsilon_1;0)$;
\item \textbf{I3}: $I^{\mathfrak m}(0;\epsilon_1 \dots \epsilon_n;1)=(-1)^nI^{\mathfrak m}(0;(1-\epsilon_n)\dots (1-\epsilon_1);1)$.
\item \label{shufflemotiter}\textbf{I4}: The \emph{shuffle product formula} $$\forall \epsilon_i,x,y\in\{0,1\}, I^{\mathfrak m}(x;\epsilon_1,\dots,\epsilon_r;y)I^{\mathfrak m}(x;\epsilon_{r+1},\dots,\epsilon_{r+s};y)=\sum_{\sigma\in\mathfrak S(r,s)}I^{\mathfrak m}(x;\epsilon_{\sigma(1)},\dots,\epsilon_{\sigma(r+s)};y)$$ where $\mathfrak S(r,s)$ is the set of $(r,s)$-shuffles introduced in \Cref{Conv}\ref{shufflesymbol}.
\end{enumerate}

\end{thm}

\begin{defi} We keep notations of \Cref{motmzvdef}. 
The symbols  $I^\m(\epsilon_0;\epsilon_1,\dots \epsilon_n;\epsilon_{n+1})$ for $\epsilon_0,\dots, \epsilon_{n+1}$ in $\{0,1\}$ are named \emph{motivic iterated integrals}.
\end{defi}

\begin{defi}We keep notations of \Cref{motmzvdef}. 
Elements of $\mathcal H_n$ are said to have weight $n$.
\end{defi}

\begin{thm}\label{permap}
    There is a well-defined map: $$\mathrm{per}:\left\{\begin{array}{lll}\mathcal H^\MT&\rightarrow &\mathbb R\\
I^{\mathfrak m}(\epsilon_0;\epsilon_1\dots\,\epsilon_n;\epsilon_{n+1})&\mapsto& I(\epsilon_0;\epsilon_1\dots\,\epsilon_n;\epsilon_{n+1}) 
\end{array}
\right.$$ which is a ring homomorphism.
\end{thm}

\begin{defi}\label{defizetam}
    Let $n_1,\dots,n_r\in\mathbb N$ with $n_r\geqslant 2$. We define the motivic multiple zeta value $\zeta^{\mathfrak m}(n_1,\dots,n_r)$ to be the element $(-1)^rI^{\mathfrak m}(0;\underbrace{10\dots0}_{n_1}\underbrace{10\dots0}_{n_2}\dots\underbrace{10\dots0}_{n_r};1)$. 
\end{defi}
\begin{rem} We keep notations of \Cref{defizetam}. Using \Cref{lemitintegr}\ref{coincide}, the period map of \Cref{permap} sends $\zeta^\m(n_1,\dots,n_r)$ to $\zeta(n_1,\dots,n_r)$ when $(n_1,\dots,n_r)$ is admissible (in the sense of \Cref{admituple}).
\end{rem}

\begin{ex}
    In particular, for all $n\geqslant 2$ we have $$\zeta^\m(n)=-I^\m(0;\underbrace{10\dots 0}_n;1)$$ and for all $a,b$ with $b\geqslant 2$, we have $$\zeta^\m(a,b)=I^\m(0;\underbrace{10\dots 0}_a\underbrace{10\dots 0}_b;1)$$ which are named double zeta values.
\end{ex}
\begin{nota}
        Let $n\geqslant 1$ be an integer. We introduce the piece of notation $f_{2n}$, which we define to be $b_nf_2^n$ for $b_n$ the Bernoulli number. 
    \end{nota}

\begin{defi}\label{defiu}We recall that the Hopf algebra $\mathcal U'$ is introduced in \Cref{shuffle}.\\ 
    Let us define a Hopf comodule $$\mathcal U=\mathcal U'\otimes_{\Q}\Q[f_2]$$ such that the coaction $\Delta$ verifies $\Delta(f_2)=f_2\otimes 1$.
\end{defi}
\begin{defi}
  Let us define a structure of algebra on $\mathcal U$ (\Cref{defiu}). 
  On $\mathcal U'$, the multiplication is given by the shuffle as in \Cref{shuffle}. We set that $f_2$ commutes with every cogenerators $f_i$ for odd $i\geqslant 3$.
\end{defi}
\begin{prop}\label{phis}
    We have an isomorphism of algebra-comodules $$\phi:\mathcal H^{\MT}\xrightarrow{\sim} \mathcal U$$ which is non-canonical. We define it so that it sends $\zeta^\m(2)$ to $f_2$ and $\zeta^\m(2n+1)$ to $f_{2n+1}$ for any positive integer $n$. It is uniquely defined and we say that $\phi$ is \emph{normalized}.
\end{prop}
  \begin{ex}\cite[Example 3.5]{Bro19} Given a word $w$ in the $f_i$'s, we have no control on the period $\mathrm{per}(\phi^{-1}(w))$ where $\mathrm{per}$ is introduced in \Cref{permap}. Let us give an example which illustrates this fact.\\ 
The minimal motive $M(f_3f_9)$ such that the word $f_3f_9$ is a motivic period of $M(f_3f_9)$ is an iterated extension of the pure motives $\Q(0)$, $\Q(3)$ and $\Q(12)$. We have:
\begin{multline*}
 \mathrm{per}(\phi^{-1}(f_3f_9))=\frac{1}{19.691}\left(
    2^43^2\zeta(5,3,2,2)-\frac{3^35.179}{2.7}\zeta(5,7)-2.3^329\zeta(7,5)
    -3.7^2\zeta(3)\zeta(9)\right.\\\left.+2^43\zeta(3)^4+2^53^311\zeta(3,7)\zeta(2)
    +2^53^231\zeta(7,3)\zeta(2)-2^43^4\zeta(3,5)\zeta(4)-2^53^2\zeta(5,3)\zeta(4)\right.\\\left.-2^33.5^2\zeta(3)^2\zeta(6)+\frac{3.128583229}{2^47.691}\zeta(12)
    \right)   
\end{multline*} modulo $\pi^{12}\Q$.\\
We remark that even though the motive $M(f_3f_9)$ is a tower of only two iterated extensions, the maximal depth of the multiple zeta values appearing is $4$ with $\zeta(5,3,2,2)$. In fact, there is no general rule to know the maximal depth appearing in this decomposition.
\end{ex}
\section{Coaction and derivation operators}\label{subsectcoact}
In this section, we describe some derivation operators defined in \cite{Bro12a}. In a first place, we exhibit different quotients of $\mathcal H^\MT$. It allows to define a simple coaction operator $D_r$ introduced in \cite[Definition 3.1]{Bro12a}. Then, we define a derivation operator $\partial_r^\phi $ as in \cite[Subsection 4.1]{Bro12b}. We explain the link between $D_r$ and $\partial_r^\phi$ in \Cref{derivcoact}.\\
We introduce a piece of notation $\zeta^?$ and $I^?$ for $?\in\{\m,\mathcal L,\mathfrak a\}$ as in \Cref{notalevels}. We will use this piece of notation throughout the section.\\

\begin{defi}\label{defL}
    We introduce the Lie coalgebra $$\mathcal L=\frac{\mathcal A_{>0}}{\mathcal A_{>0}\A_{>0}}$$ of indecomposable elements of $\mathcal A$.
\end{defi}
\begin{nota}\label{notalevels} Let $\textbf{n} $ be a tuple of positive integers.
    The motivic multiple zeta values or the motivic iterated integrals can be written in three levels. The first one is the comodule $\mathcal H^\MT$, the second one is the Hopf algebra $\A^\MT$ and finally we can project them onto $\mathcal L$. To distinguish where we consider them, we introduce the following notations $$\left\{\begin{array}{cccccc}
        \mathcal H^\MT_{>0}&\rightarrow &\A^\MT_{>0}&\rightarrow &\mathcal L_{>0}\\
        \zeta^\m(\textbf{n})&\mapsto&\zeta^\mathfrak a(\textbf{n})&\mapsto&\zeta^\mathcal L(\textbf{n})
        \end{array}
        \right.$$ where the map from $\mathcal A^\MT_{>0}$ to $\mathcal L_{>0}$ is the projection. We recall that  $\mathcal H^\MT$ is isomorphic as a comodule to $\mathcal A^\MT\otimes_\Q\Q[f_2]$ as stated in \Cref{motmzvdef}\ref{isoheta}. The map of comodules from $\mathcal H_{>0}$ to $\mathcal A^\MT_{>0}$ is induced by the map from $\mathcal A^\MT\otimes_\Q\Q[f_2]$ to $\mathcal A^\MT$ sending $\zeta^\m(2)$ to $0$. 
        We will use the same indexes $\mathfrak m$, $\mathfrak a$ and $\mathcal L$ for motivic iterated integrals.
        
\end{nota}
\begin{defi}\label{coact}
    Let $r\geqslant 3$ be an odd integer and $N\geqslant r+2$ a positive integer. We consider the operator $$D_r:\mathcal H_N\rightarrow\mathcal L_r\otimes_\Q \mathcal H_{N-r}$$ which is defined by the composition of the part $(r,N-r)$ of the coaction and of the projection onto $\mathcal L$ as follows $$D_r:\mathcal H_N\xrightarrow{\Delta_{r,N-r}}\A_r\otimes\mathcal H_{N-r}\rightarrow\mathcal L_r\otimes \mathcal H_{N-r}.$$
    More concretely, it sends $I^{\m}(\epsilon_0;\epsilon_1\dots \epsilon_N;\epsilon_{N+1})$ to $$D_rI^{\m}(\epsilon_0;\epsilon_1\dots \epsilon_N;\epsilon_{N+1})=\sum_{p=0}^{N-r}I^{\mathcal L}(\epsilon_p;\epsilon_{p+1}\dots \epsilon_{p+r};\epsilon_{p+r+1})\otimes I^{\m}(\epsilon_0;\epsilon_1\dots \epsilon_{p}\,\epsilon_{p+r+1}\dots \epsilon_N;\epsilon_{N+1}).$$
    \end{defi}
    \begin{rem}
        We could have defined $D_r$ only as the part $(r,N-r)$ of the coaction without the projection on $\mathcal L$ (\Cref{defL}). Yet, this projection simplifies the expression, since we only consider indecomposable elements, id est elements of $\mathcal A_{>0}$ which are not a product of elements of $\mathcal A_{>0}$. We will see that we still have the needed information after this projection in \Cref{derivcoact}.
    \end{rem}
    
   \begin{ex}\label{zeta35d3}
       Let us consider $\zeta^\m(3,5)=I^\m(0;10010000;1)$. Using the formula of the coaction in \Cref{coact}, we have $$D_3\zeta^\m(3,5)=I^\mathcal L(0;100;1)\otimes I^\m(0;10000;1)+I^\mathcal L(1;001;0)\otimes I^\m(0;10000;1).$$ Applying \Cref{motmzvdef}\ref{i2} we obtain $I^\mathcal L(1;001;0)=-I^\mathcal L(0;100;1)$. Hence, $$D_3\zeta^\m(3,5)=\zeta^\mathcal L(3)\otimes \zeta^\m(5)-\zeta^\mathcal L(3)\otimes \zeta^\m(5)=0.$$
   \end{ex}
   
       \begin{ex}\label{d3zeta53}
    We can compute $$D_3\zeta^\m(5,3)=D_3I^\m(0;10000100;1)=I^\mathcal L(0;100;1)\otimes I^\m(0;10000;1)=\zeta^\mathcal L(3)\otimes \zeta^\m(5)$$ by applying the formula of the coaction (\Cref{coact}).
\end{ex}   
    \begin{ex}\label{zeta35d5}
    Let us compute $D_5\zeta^\m(3,5).$ We obtain $$D_5\zeta^\m(3,5)=I^\mathcal L(1;00100;0)\otimes I^\m(0;100;1)+I^\mathcal L(0;10000;1)\otimes I^\m(0;100;1)$$ by applying the formula of the coaction of \Cref{coact}.
    We have $$I^\mathcal L(1;00100;0)=-I^\mathcal L(0;00100;1)$$ by using \Cref{motmzvdef}\ref{i2}.
    Finally, we obtain $$-I^\mathcal L(0;00100;1)=-(-1)^2\binom{4}{2}I^\mathcal L(0;10000;1)=-6\,I^\mathcal L(0;10000;1)$$ by using \Cref{formul}.
        Hence, the following equality $$D_5\zeta^\m(3,5)=-6\zeta^\mathcal L(5)\otimes\zeta^\m(3)+\zeta^\mathcal L(5)\otimes\zeta^\m(3)=-5\zeta^\mathcal L(5)\otimes\zeta^\m(3)$$ holds.
    \end{ex}
    \begin{ex}\label{zeta73d7}
        Let us compute $D_7\zeta^\m(7,3)$. We apply the formula of the coaction from \Cref{coact}: \begin{multline*}
       D_7\zeta^\m(7,3)=D_7I^\m(0;\underbrace{10\dots 0}_7100;1)=I^\mathcal L(0;\underbrace{10\dots0}_7;1)\otimes I^\m(0;100;1)\\+I^\mathcal L(1;\underbrace{0\dots01}_7;0)\otimes I^\m(0;100;1)+I^\mathcal L(0;0000100;1)\otimes I^\m(0;100;1) \end{multline*} where we have $I^\mathcal L(1;\underbrace{0\dots01}_7;0)=-I^\mathcal L(0;\underbrace{10\dots0}_7;1)$ by \textbf{I2} (\Cref{motmzvdef}\ref{i2}) and $$I^\mathcal L(0;0000100;1)=(-1)^4\binom{6}{4}I^\mathcal L(0;\underbrace{10\dots0}_7;1)=15I^\mathcal L(0;\underbrace{10\dots0}_7;1)$$ by \Cref{formul}. Hence, we obtain $$D_7\zeta^\m(7,3)=15\zeta^\mathcal L(7)\otimes\zeta^\m(3).$$
    \end{ex}
 \begin{lemma}\label{formul}\cite[Equation 4.12]{Bro12b} Let $k$ and $n$ be positive integers. 
        We have the following relation: $$I^\m(0;\underbrace{0\dots0}_k1\underbrace{0\dots0}_{2n-k};1)=(-1)^{k}\binom{2n}{k}I^\m(0;\underbrace{10\dots 0}_{2n+1};1)=(-1)^{k+1}\binom{2n}{k}\zeta^\m(2n+1).$$
    \end{lemma}

\begin{defi}\label{deriv}
    Let $r\geqslant 3$ be an odd integer. We define the linear operator $$\partial_r:\left\{\begin{array}{clc}
    \mathcal U'&\rightarrow &\mathcal U'\\
    f_{i_1}\dots f_{i_s}&\mapsto& \left\{\begin{array}{ll}
    f_{i_2}\dots f_{i_s} &\mathrm{if}\;i_1=r\\
    0&\mathrm{otherwise} .
    \end{array}
    \right.
    \end{array}
    \right.$$
We extend $\partial_r$ to an operator on $\mathcal U$ by setting $\partial_r(f_2)=0$.
    \end{defi}
    \begin{rem}
         The operator $\partial_r$ of \Cref{deriv} is a derivation for the shuffle product. 
    \end{rem}
   
    \begin{defi}\label{derivphi}Let $\phi:\mathcal H^\MT\xrightarrow{\sim} \mathcal U$ be a normalized isomorphism as in \Cref{phis}.
        Let $r\geqslant 3$ be an odd integer. We define the following derivation operator on $\mathcal H^\MT$: $$\partial_r^\phi:=\phi^{-1}\circ\partial_r\circ\phi:\mathcal H^\MT\rightarrow\mathcal H^\MT.$$
    \end{defi}
   
    \begin{rem}\label{polynoncom}\cite{Bro12b}
    Let $N$ be a positive odd integer and let $F$ be an element of the $N$-th graded part $\mathcal U_N$ of $\mathcal U$ (from \Cref{defiu}). The underlying vector space of $\mathcal U_N$ admits a basis given by words in the $f_i$'s of total weight $N$. Hence, the element $F$ can be uniquely written as a sum $$F=\sum_{1\leqslant i<\left\lfloor\frac N 2\right\rfloor}f_{2i+1}v_{N-2i-1}+cf_N$$ where $v_{j}$ is in $\mathcal U_j$ and $c$ is in $\Q$. 
    \end{rem}
    \begin{defi}\label{cr}
    Let $a$ be a positive integer. Let us define the $\Q$-linear map $$c_{2n+1}:\mathcal U_{2n+1}\rightarrow \Q$$  which associates to an element $\xi$ in $\mathcal U_{2n+1}$ the rational coefficient in front of $f_{2n+1}$. 
    \end{defi}
    \begin{rem}
        If we see $\xi$ as a polynomial in the non-commutative variables $f_i$ with the same notations as in \Cref{polynoncom}, then $c_{2n+1}(\xi)$ is the coefficient $c$.
    \end{rem}
   
    \begin{rem}\label{derivcoact} Let us recall that $c_{2n+1}$ is introduced in \Cref{cr}. 
        By definition, we have $$\partial^\phi_{2n+1}=((c_{2n+1}\circ\phi)\otimes id)\circ D_{2n+1}.$$ 
    \end{rem}
    \begin{rem}\label{derivutile} Let $a$, $b$ be positive integers with $b\geqslant 2$. Let us denote the total weight $a+b$ by $N$. 
        If we want to compute the image of the motivic double zeta value $\zeta^\m(a,b)$ under the normalized $\phi$ introduced in \Cref{phis}, it is relevant to compute the image of $\zeta^\m(a,b)$ under $\partial_r^\phi$ for $r$ an odd integer bigger than $1$. Indeed, by definition we have the following equalities.
        \begin{enumerate}
        \item If $N$ is even, then $$\phi(\zeta^\m(a,b))=\sum_{r=3,\,\mathrm{odd}}^{N-2}f_r\partial_r^\phi(\zeta^\m(a,b))$$ modulo $f_2^{\frac {N} 2}$. Since the derivation is zero when evaluated in $f_2$, it does not compute the coefficient in front of $f_2^{\frac N 2}$.
        \item If $N$ is odd, then $$\phi(\zeta^\m(a,b))=\sum_{r=3,\,\mathrm{odd}}^{N-2}f_r\partial_r^\phi(\zeta^\m(a,b))+c_N^\phi(\zeta^\m(a,b))f_N.$$
        \end{enumerate}
        
        \Cref{derivcoact} gives us a formula for the derivation operators. We will compute the coaction on double zeta values in \Cref{subsectcoactdouble}. Then, using \Cref{derivcoact}, we will deduce a formula for derivation operators in \Cref{subsectwritingfalpha}.
    \end{rem}

\section{Computation of the coaction for double zeta values}\label{subsectcoactdouble}
In this section, we compute the action of the operator $D_r$ on double zeta values $\zeta^\m(a,b)$ for $a$ and $b$ positive integers with $b\geqslant 2$ and $r$ an odd integer that is bigger than $1$. We work with disjunction of cases depending on $r$, $a$ and $b$. \\
We recall the piece of notation $\zeta^?$ and $I^?$ for $?\in\{\m,\mathcal L,\mathfrak a\}$ introduced in \Cref{notalevels}.
\begin{thm} \label{coactdouble} Let $a$ and $b$ be positive integers with $b\geqslant 2$.  Let us denote the total weight $a+b$ by $N$.
    \begin{enumerate}
        \item \label{r=nb} Let us suppose that $a< b$ and $a$ odd. 
The following holds: $$D_a\zeta^{\m}(a,b)=0.$$
\item \label{r=na} Let us suppose that $a\geqslant b$ and $a$ odd.
The following holds: $$D_a\zeta^{\m}(a,b)=(-1)^{b+1}\binom{a-1}{b-1}\,\zeta^{\mathcal L}(a)\otimes\zeta^{\m}(b).$$
        \item \label{lem1}
    Let $r$ be an odd integer with $a<r\leqslant N-2$. The following holds: $$D_r \zeta^{\m}(a,b)=\left((-1)^{b+1}\binom{r-1}{b-1}\mathds{1}_{\{b\leqslant r\}}+(-1)^{a}\binom{r-1}{a-1}\right)\zeta^{\mathcal L}(r)\otimes\zeta^{\m}(N-r).$$
    \item \label{lem2}
      For $r$ an odd integer with $3\leqslant r<a$, we obtain: $$ D_r \zeta^{\m}(a,b)=\left((-1)^{b+1}\binom{r-1}{b-1}\,\zeta^{\mathcal L}(r)\otimes\zeta^{\m}(N-r)\right)\mathds{1}_{\{b\leqslant r\}}.$$
    \end{enumerate}
\end{thm}
\begin{rem}
 Notice that all possible values of $r$ (odd values with $3\leqslant r\leqslant N-2$) are treated in cases of \Cref{coactdouble}.
\end{rem}
\begin{rem}\label{drrem}
    We recall the formula for the operator $D_r$: \begin{equation} \label{dr}
    D_rI^{\m}(\epsilon_0;\epsilon_1\dots \epsilon_N;\epsilon_{N+1})=\sum_{p=0}^{N-r}I^{\mathcal L}(\epsilon_p;\epsilon_{p+1}\dots \epsilon_{p+r};\epsilon_{p+r+1})\otimes I^{\m}(\epsilon_0;\epsilon_1\dots \epsilon_p \,\epsilon_{p+r+1}\,\epsilon_{p+r+2}\dots\,\epsilon_N;\epsilon_{N+1})
    \end{equation} for any $(\epsilon_i)_{0\leqslant i\leqslant N+1}$ in $\{0,1\}^{N+2}$ (Definition \ref{coact}).
\end{rem}
\begin{proof}
\begin{enumerate}
    \item Let us consider terms in the sum depending on $p$, where $p$ is the variable of the sum \eqref{dr} in \Cref{drrem}. \\

For $p=0$, we obtain the term \begin{equation}\label{inutile}
    I^{\mathcal L}(0;\underbrace{10\dots0}_a;1)\otimes I^{\m}(0;\underbrace{10\dots0}_b;1)=\zeta^\mathcal L(a)\otimes\zeta^\m(b).
    \end{equation}
    
    For $p=1$, we obtain the term $I^{\mathcal L}(1;\underbrace{0\dots 0}_{a-1}1;0)\otimes I^{\m}(0;\underbrace{10\dots0}_{b};1)$.\\ by \Cref{motmzvdef}\ref{i2}, we have $$I^{\mathcal L}(1;\underbrace{0\dots0}_{a-1}1;0)=(-1)^aI^{\mathcal L}(0;1\underbrace{0\dots0}_{a-1};1).$$ Since $a$ is supposed to be an odd integer, we have: \begin{equation}\label{inutilebise}
   I^{\mathcal L}(1;\underbrace{0\dots0}_{a-1}1;0)\otimes I^{\m}(0;\underbrace{10\dots0}_{b};1)=(-1)^a\zeta^\mathcal L(a)\otimes\zeta^\m(b)=-\zeta^\mathcal L(a)\otimes\zeta^\m(b).
    \end{equation}
  Both terms \eqref{inutile} and \eqref{inutilebise} compensate in the sum.\\

 Now, let us show that any other term of the sum is zero.
If $p$ is in $\{ 2,\dots,a-1\}$, then the tensor product appearing in the sum at $p$ is \begin{equation}\label{zerotermsp}
I^\mathcal L(0;\underbrace{0\dots 0}_{a-p} 1\underbrace{0\dots 0}_{p-1};0)\otimes I^\m(0;1\underbrace{0\dots 0}_{p-1}\underbrace{0\dots 0}_{b-p};1)=0 \end{equation}
because $I^\mathcal L(0;\underbrace{0\dots 0}_{a-p} 1\underbrace{0\dots 0}_{p-1};0)$ is zero.\\
For $p=a+1$, we obtain the term \begin{equation}\label{zero}I^\mathcal L(1;\underbrace{0\dots 0}_a;\varepsilon)\otimes I^\m(0;\underbrace{10\dots0}_b;1)=0\end{equation} where $$\varepsilon=\left\{\begin{array}{ll} 1&\mathrm{if\,} b=a+1\\
0&\mathrm{otherwise}.
\end{array}
\right.$$ In both cases, we have $I^\mathcal L(1;\underbrace{0\dots 0}_a;\varepsilon)=0$.\\

For $p\in\{ a+2,\dots,b-1\}$, we have the term $$I^\mathcal L(0;\underbrace{0\dots 0}_a;0)\otimes I^\m(0;\underbrace{10\dots0}_a\underbrace{10\dots 0}_{b-a};1)=0$$ because $I^\mathcal L(0;\underbrace{0\dots 0}_a;0)=0$.\\
For $p=b$, we have already treated the case when $b=a+1$ in \eqref{zero}, so let us suppose that $b>a+1$. We obtain the term $$I^\mathcal L(0;\underbrace{0\dots0}_a;1)\otimes I^\m(0;\underbrace{10\dots0}_a\underbrace{10\dots 0}_{b-a};1)=0.$$
\begin{rem}
    In the proof of \Cref{coactdouble}\ref{r=nb}, we have reviewed all reasons why a term in the sum of \Cref{dr} in \Cref{drrem} could be zero. From now on, we focus on non-zero terms of this sum.
\end{rem}
\item We compute $D_aI^\m(0;\underbrace{10\dots0}_a\underbrace{1\dots0}_b;1)$. 
With the same arguments as in the proof of  \Cref{coactdouble}\ref{r=nb}, we can see that the non-zero terms that might appear in the sum correspond to $p$ in $\{0,1,N-a\}$ where we recall that $N-a$ is $b$.\\

For $p=0$ or $1$, we otain the same terms \eqref{inutile} and \eqref{inutilebise} as in the proof of \Cref{coactdouble}\ref{r=nb} which compensate each other in the sum.\\

    For $p=N-a=b$, we obtain $$I^{\mathcal L}(0;\underbrace{0\dots0}_{a-b}\underbrace{10\dots0}_{b};1)\otimes I^{\m}(0;\underbrace{10\dots0}_b;1).$$ Using \Cref{formul}, we know that this is equal to $$(-1)^{a-b}\binom{a-1}{a-b}\,I^{\mathcal L}(0;\underbrace{10\dots0}_a;1)\otimes I^{\m}(0;\underbrace{10\dots0}_b;1).$$ Hence, we obtain: $$D_a\zeta^{\m}(a,b)=(-1)^{a-b}\binom{a-1}{a-b}\,\zeta^{\mathcal L}(a)\otimes\zeta^{\m}(b).$$
     We recall that $a$ is odd and thus $(-1)^{a-b}=(-1)^{b+1}$. We obtain the term: \begin{equation}\label{eq2nmodd}D_a\zeta^{\m}(a,b)=(-1)^{b+1}\binom{a-1}{b-1}\,\zeta^{\mathcal L}(a)\otimes\zeta^{\m}(b).
    \end{equation} We sum the terms \eqref{inutile}, \eqref{inutilebise} and \eqref{eq2nmodd} to obtain: $$D_a\zeta^{\m}(a,b)=\zeta^\mathcal L(a)\otimes \zeta^\m(b)-\zeta^\mathcal L(a)\otimes \zeta^\m(b)+(-1)^{b+1}\binom{a-1}{b-1}\,\zeta^{\mathcal L}(a)\otimes\zeta^{\m}(b).$$ This concludes the proof, since terms \eqref{inutile} and \eqref{inutilebise} compensate each other.
   
    \item The only values of $p$ for which terms in \Cref{drrem}\eqref{dr} might be non-zero are for $p\in\{1,N-r\}$. To see it, we use the same kind of reasoning as in the beginning of the proof of \Cref{coactdouble}\ref{r=nb}.\\
    
    For $p=1$, we obtain the term $I^{\mathcal L}(1;\underbrace{0\dots0}_{a-1}1\underbrace{0\dots0}_{r-a};0)\otimes I^{\m}(0;\underbrace{10\dots0}_{N-r};1)$. Using the inversion formula \textbf{I2} (\Cref{motmzvdef}\ref{i2}) and the fact that $r$ is odd, we know that $$I^{\mathcal L}(1;\underbrace{0\dots0}_{a-1}1\underbrace{0\dots0}_{r-a};0)=-I^{\mathcal L}(0;\underbrace{0\dots0}_{r-a}1\underbrace{0\dots0}_{a-1};1).$$ by \Cref{formul}, it is equal to $$-(-1)^{r-a}\binom{r-1}{r-a}\,I^{\mathcal L}(0;\underbrace{10\dots0}_{r};1).$$ We obtain $$(-1)^{r-a+1}\binom{r-1}{r-a}\,\zeta^{\mathcal L}(r)\otimes \zeta^{\m}(N-r)=(-1)^a\binom{r-1}{a-1}\,\zeta^{\mathcal L}(r)\otimes \zeta^{\m}(N-r)$$ since $r$ is odd.
    \newline
    For $p=N-r$, there is a non-zero term if and only if $b\leqslant r$. In this case, there is the term 
    \begin{equation}\label{mleqr}
    I^{\mathcal L}(0;\underbrace{0\dots0}_{r-b}\underbrace{10\dots 0}_{b};1)\otimes I^{\m}(0;\underbrace{10\dots0}_{N-r};1)
    \end{equation}
    which equals $$(-1)^{r-b}\binom{r-1}{r-b}\,I^{\mathcal L}(0;\underbrace{10\dots0}_{r};1)\otimes I^{\m}(0;\underbrace{10\dots0}_{N-r};1)$$ by \Cref{formul}. If we express it with motivic multiple zeta values, we obtain: \begin{equation}\label{calcpourm}
    (-1)^{r-b}\binom{r-1}{r-b}\,\zeta^{\mathcal L}(r)\otimes \zeta^{\m}(N-r)=(-1)^{b+1}\binom{r-1}{b-1}\,\zeta^{\mathcal L}(r)\otimes \zeta^{\m}(N-r)
    \end{equation} since $r$ is odd.\\
    Hence, we obtain the following formula: $$D_r \zeta^{\m}(a,b)=\left((-1)^{b+1}\binom{r-1}{b-1}\,\zeta^{\mathcal L}(r)\otimes\zeta^{\m}(N-r)\right)\mathds{1}_{\{b\leqslant r\}}+(-1)^{a}\binom{r-1}{a-1}\,\zeta^{\mathcal L}(r)\otimes\zeta^{\m}(N-r).$$ 
    \item   The only non-zero term appears for $p=N-r$ and only in the case $b\leqslant r$. In this case the computation is exactly the same as above in the proof of \Cref{coactdouble}\ref{lem1} (see Equation \eqref{mleqr}) and we obtain the same term as in \Cref{calcpourm}.\qedhere\end{enumerate}\end{proof}

    \begin{ex}
        In Example \ref{zeta73d7}, we have computed $D_7\zeta^\m(7,3)$ and we have obtained $$D_7\zeta^\m(7,3)=(-1)^4\binom{7-1}{3-1}\,\zeta^\mathcal L(7)\otimes \zeta^\m(3)=15\,\zeta^\mathcal L(7)\otimes \zeta^\m(3)$$ which is the same as in \Cref{coactdouble}\ref{r=na}.
    \end{ex}
    
\begin{ex}
        We have proved in Example \ref{zeta35d3} that $$D_3\zeta^\m(3,5)=I^\mathcal L(0;100;1)\otimes I^\m(0;10000;1)+I^\mathcal L(1;001;0)\otimes I^\m(0;10000;1)=0.$$ It coincides with the result of \Cref{coactdouble}\ref{r=nb}. We can see the compensation between terms \eqref{inutile} and \eqref{inutilebise} explained in the proof of \Cref{coactdouble}\ref{r=nb}. More precisely $I^\mathcal L(0;100;1)\otimes I^\m(0;10000;1)$ is the term \eqref{inutile} and $I^\mathcal L(1;001;0)\otimes I^\m(0;10000;1)$ is the term \eqref{inutilebise}.
    \end{ex}

    \begin{ex}
        In Example \ref{zeta35d5}, we have obtained $$D_5\zeta^\m(3,5)=-6\zeta^\mathcal L(5)\otimes\zeta^\m(3)+\zeta^\mathcal L(5)\otimes\zeta^\m(3)=-5\zeta^\mathcal L(5)\otimes\zeta^\m(3)$$ which is exactly $$(-1)^{5+1}\binom 4 4\,\zeta^\mathcal L(5)\otimes\zeta^\m(3)+(-1)^3\binom 4 2\,\zeta^\mathcal L(5)\otimes\zeta^\m(3)$$ as in \Cref{coactdouble}\ref{lem1}.
    \end{ex}
 \begin{ex}
In \Cref{d3zeta53}, we have obtained $$D_3\zeta^\m(5,3)=\zeta^\mathcal L(3)\otimes\zeta^\m(5)$$ which is exactly $$(-1)^4\binom {3-1}{3-1}\zeta^\mathcal L\otimes\zeta^\m(5)$$ as stated in \Cref{coactdouble}\ref{lem2}.
 \end{ex}
    \section{Decomposing double zeta values in words in the cogenerators $f_i$}\label{subsectwritingfalpha}
The goal of this section is to write $\phi(\zeta^\m(a,b))$ in terms of the cogenerators $f_i$ for the normalized isomorphism $\phi:\mathcal H^\MT\xrightarrow{\sim}\mathcal U$ that was fixed in \Cref{phis}. To do so, we use the derivation operators $\partial_r^\phi$ introduced in \Cref{derivphi} for any odd $r$ bigger than $1$, as explained in \Cref{derivutile}. Indeed, the operators $\partial_r^\phi$ describe the decomposition of $\phi(\zeta^\m(a,b))$ in words in the $f_i$'s. By \Cref{derivcoact}, we know that we can compute the operators $\partial_r^\phi$ using the expression of the coaction $D_r$ on double zeta values, which is computed in \Cref{subsectcoactdouble}. This method is used by Brown in \Cite{Bro12b}.\\
By computing $\phi(\zeta^\m(a,b))$, we obtain new relations between motivic multiple zeta values as explained in \Cref{perconjbb}. When $a$ and $b$ do not have the same parity, we obtain a lift in the motivic context of relations between multiple zeta values from \cite[Section 5]{Borwein}, which goes in the sense of the period conjecture.\\
For $r\geqslant 3$ an odd integer, let us denote $c_r\circ\phi$ by $c_r^\phi$ where $c_r$ is defined in \Cref{cr}. We recall the piece of notation $\zeta^?$ and $I^?$ for $?\in\{\m,\mathcal L,\mathfrak a\}$, introduced in \Cref{notalevels}.\\

\begin{thm}\label{thmderiv} Let $a$ and $b$ be positive integers with $b\geqslant2$. Let us denote the total weight $a+b$ by $N$. 
    \begin{enumerate}
   
\item \label{cor2} Let us suppose that $a< b$ and $a$ odd.
The following holds: $$\partial_a^\phi\zeta^{\m}(a,b)=0.$$
\item \label{cor1} Let us suppose that $a\geqslant b$ and $a$ odd.
The following holds: $$\partial_a^\phi\zeta^{\m}(a,b)=(-1)^{b+1}\binom{a-1}{b-1}\zeta^\m(b).$$
\item \label{cor3}  We fix $3\leqslant r\leqslant N-2$ where $r$ is an odd integer which is different than $a$. Then we have the following formula: $$\partial_r^\phi  \zeta^{\m}(a,b)=\left((-1)^{b+1}\binom{r-1}{b-1}\,\mathds{1}_{\{b\leqslant r\}}+(-1)^{a}\binom{r-1}{a-1}\mathds 1_{\{a+1\leqslant r\}}\right)\,\zeta^{\m}(N-r).$$

\end{enumerate}
\end{thm}
\begin{proof} In \Cref{derivcoact}, we have stated $$\partial_r^\phi=(c_r^\phi\otimes id)\circ D_r$$ for all odd integer $r$ bigger than $3$. 
    \begin{enumerate}
    \item Using \Cref{coactdouble}\ref{r=nb}, the equality $$D_a\zeta^{\m}(a,b)=0$$ holds. Hence, using \Cref{derivcoact}, we have $$\partial_a^\phi\zeta^\m(a,b)=((c_a\circ\phi)\otimes id)(D_a\zeta^\m(a,b))=0$$ as announced.
    \item Using \Cref{coactdouble}\ref{r=na}, the equality 
    $$D_a\zeta^{\m}(a,b)=(-1)^{b+1}\binom{a-1}{b-1}\,\zeta^{\mathcal L}(a)\otimes\zeta^{\m}(b)$$ holds. Hence, using \Cref{derivcoact}, we have $$(c_a^\phi\otimes id)(D_a\zeta^{\m}(a,b))=\left(c_a^\phi\left((-1)^{b+1}\binom{a-1}{b-1}\zeta^\m(a)\right)\right)\otimes \zeta^\m(b)$$ where $c_a^\phi$ gives the coefficient in front of $\zeta^\m(a)$. Hence we obtain $$\partial_a^\phi\zeta^{\m}(a,b)=(-1)^{b+1}\binom{a-1}{b-1}\zeta^\m(b)$$ as announced.
    
    \item  Using \Cref{coactdouble}\ref{lem1} and \ref{lem2}, the equality $$D_r\zeta^\m(a,b)=\left((-1)^{b+1}\binom{r-1}{b-1}\,\mathds{1}_{\{b\leqslant r\}}+(-1)^{a}\binom{r-1}{a-1}\,\mathds 1_{\{a+1\leqslant r\}}\right)\zeta^{\mathcal L}(r)\otimes\zeta^{\m}(N-r)$$ holds. Hence, using \Cref{derivcoact}, we have \begin{multline*}\partial^\phi_r\zeta^\m(a,b)=((c_r\circ\phi)\otimes id)\left(\left((-1)^{b+1}\binom{r-1}{b-1}\,\mathds{1}_{\{b\leqslant r\}}  +(-1)^{a}\binom{r-1}{a-1}\,\mathds 1_{\{a+1\leqslant r\}}\right)\zeta^{\mathcal L}(r)\right)\otimes\zeta^{\m}(N-r)\end{multline*}
     which is exactly $$\partial^\phi_r\zeta^\m(a,b)=c_r\left((-1)^{b+1}\binom{r-1}{b-1}\,\zeta^\m(r)\mathds{1}_{\{b\leqslant r\}}+(-1)^{a}\binom{r-1}{a-1}\,\zeta^\m(r)\mathds 1_{\{a+1\leqslant r\}}\right)\otimes \zeta^\m({N-r})$$ and we finally obtain $$\partial^\phi_r\zeta^\m(a,b)=\left((-1)^{b+1}\binom{r-1}{b-1}\mathds{1}_{\{b\leqslant r\}}+(-1)^{a}\binom{r-1}{a-1}\mathds 1_{\{a+1\leqslant r\}}\right)\zeta^\m({N-r})$$ as announced.\qedhere
    \end{enumerate}
\end{proof}
\begin{lemma}\label{nmdiffparities}
Let $a$ and $b$ be integers of different parities with $b\geqslant 2$. We write $N$ for the total weight $a+b$. We have the following formula:  $$\phi(\zeta^{\mathfrak m}(a,b))=\sum_{r=b,\,r\,\mathrm{odd}}^{N-2}\left((-1)^{b+1}\binom{r-1}{b-1}\,f_rf_{N-r}\right)+\sum_{r=a+1,\,r\,\mathrm{odd}}^{N-2}\left((-1)^{a}\binom{r-1}{a-1}\,f_rf_{N-r}\right)+c_{N}^\phi(\zeta^\m(a,b))f_N$$ with $c_{N}^\phi(\zeta^\m(a,b))$ being $(c_N\circ\phi)(\zeta^\m(a,b))$ where $c_N$ is introduced in \Cref{cr}.
\end{lemma}
\begin{proof} 
We can decompose $\zeta^\m(a,b)$ as $$\phi(\zeta^m(a,b))=\sum_{r= 3,\,\mathrm{odd}}^{N-2}f_{r}\partial_r(\phi(\zeta^\m(a,b)))+(c_N\circ\phi)(\zeta^\m(a,b))$$ by definition of the derivation operator (Definitions \ref{deriv} and \ref{derivphi}). 

Using \Cref{thmderiv} \ref{cor2}, \ref{cor1} and \ref{cor3}, we can obtain $$\sum_{r= 3,\,\mathrm{odd}}^{N-2}f_r\,\partial_r(\phi(\zeta^\m(a,b)))=\sum_{r=b,\,r\,\mathrm{odd}}^{N-2}\left((-1)^{b+1}\binom{r-1}{b-1}\,f_rf_{N-r}\right)+\sum_{r=a+1,\,r\,\mathrm{odd}}^{N-2}\left((-1)^{a}\binom{r-1}{a-1}\,f_rf_{N-r}\right).$$
\end{proof}
  
\begin{thm}\label{ecrfalpha}Let $a$ and $b$ be positive integers with $b\geqslant2$. Let us denote the total weight $a+b$ by $N$.
\begin{enumerate}
\item \label{writingfalphanmodd} Let $a$ and $b$ be both odd. The image of the motivic double zeta value $\zeta^{\mathfrak m}(a,b)$ under $\phi$ is: $$\phi(\zeta^\m(a,b))=\sum_{r=3,\,r\,\mathrm{odd}}^{N-2}\left(\binom{r-1}{b-1}-\binom{r-1}{a-1}\right)f_rf_{N-r}+\mathds 1_{\{a\geqslant 3\}}\,f_af_b$$ modulo $f_2^{\frac N 2}$.
    \item \label{writingfalphanmeven} Let $a$ and $b$ be both even. The image of the motivic double zeta value $\zeta^{\mathfrak m}(a,b)$ under $\phi$ is: $$\phi(\zeta^\m(a,b))=\sum_{r=3,\,r\,\mathrm{odd}}^{N-2}\left(\binom{r-1}{a-1}-\binom{r-1}{b-1}\right)f_rf_{N-r}$$ modulo $f_2^{\frac N 2}$.
\item \label{neven}
Let us suppose that $a$ is even and $b$ is odd. The image of the motivic double zeta value $\zeta^{\mathfrak m}(a,b)$ under $\phi$ is:  $$\phi(\zeta^{\mathfrak m}(a,b))=\sum_{r=3,\,r\,\mathrm{odd}}^{N-2}\left(\binom{r-1}{b-1}+\binom{r-1}{a-1}\,\right)f_rf_{N-r} -\frac 1 2\left(\binom N a+1\right)f_N.$$
\item \label{nodd}
Let us suppose that $a$ is odd and $b$ is even. The image of the motivic double zeta value $\zeta^{\mathfrak m}(a,b)$ under $\phi$ is:  $$\phi(\zeta^{\mathfrak m}(a,b))=-\sum_{r=3,\,r\,\mathrm{odd}}^{N-2}\left(\binom{r-1}{b-1}+\binom{r-1}{a-1}\,\right)f_rf_{N-r} +\mathds 1_{\{a\geqslant 3\}}\,f_af_b+\frac 1 2\left(\binom N a-1\right)f_N.$$
    \end{enumerate}
\end{thm}
   
    \begin{proof}
    \begin{enumerate}
         \item By definition, we have $$\phi(\zeta^\m(a,b))=\sum_{r=3,r\,\mathrm{odd}}^{N-2}f_r\,\partial_r(\phi(\zeta^\m(a,b))$$ modulo $f_2^{\frac N 2}$ (\Cref{derivutile}). 
        Let us suppose that $a\geqslant b$. We use \Cref{thmderiv} \ref{cor1} and \ref{cor3} to obtain $$\phi(\zeta^\m(a,b))=\mathds 1_{\{a\geqslant 3\}}\,(-1)^{b+1}\binom{a-1}{b-1}\,f_af_{b}+\sum_{r=b,\,r\neq a,\,r\,\mathrm{odd}}^{N-2}(-1)^{b+1}\binom{r-1}{b-1}\,f_rf_{N-r}+\sum_{r=a+1,r\,\mathrm{odd}}^{N-2}
        (-1)^{a}\binom{r-1}{a-1}\,f_rf_{N-r}$$ modulo $f_2^{\frac N 2}$. We use the fact that $a$ and $b$ are odd to obtain $$\phi(\zeta^\m(a,b))=\sum_{r=b,\,r\,\mathrm{odd}}^{N-2}\binom{r-1}{b-1}\,f_rf_{N-r}-\sum_{r=a+1,\,r\,\mathrm{odd}}^{N-2}\binom{r-1}{a-1}\,f_rf_{N-r}\quad\mathrm{modulo}\;f_2^{\frac N 2}$$ which can be rewritten as in the theorem.\\
        Let us suppose that $a< b$. We use \Cref{thmderiv}\ref{cor2} and \ref{cor3} to obtain 
      $$
       \phi(\zeta^\m(a,b))=\sum_{r=b,\,r\,\mathrm{odd}}^{N-2}(-1)^{b+1}\binom{r-1}{b-1}\,f_rf_{N-r}+\sum_{r=a+1,r\,\mathrm{odd}}^{N-2}
        (-1)^{a}\binom{r-1}{a-1}\,f_rf_{N-r}
       $$
        modulo $f_2^{\frac N 2}$. 
        We use the fact that $a$ and $b$ are odd to obtain $$\phi(\zeta^\m(a,b))=\sum_{r=b,\,r\,\mathrm{odd}}^{N-2}\binom{r-1}{b-1}\,f_rf_{N-r}-\sum_{r=a+1,\,r\,\mathrm{odd}}^{N-2}\binom{r-1}{a-1}\,f_rf_{N-r}\quad\mathrm{modulo}\;f_2^{\frac N 2}$$  which can be rewritten as in the theorem.
        \item  It is almost the same proof as \Cref{ecrfalpha}\ref{writingfalphanmodd} but the signs change because now $a$ and $b$ are even instead of being odd. Let us be more precise.\\
        By definition, we have $$\phi(\zeta^\m(a,b))=\sum_{r=3,r\,\mathrm{odd}}^{N-2}f_r\,\partial_r(\phi(\zeta^\m(a,b)))$$ modulo $f_2^{\frac N 2}$ (\Cref{derivutile}). 
         We use \Cref{thmderiv} \ref{cor3} to obtain $$\phi(\zeta^\m(a,b))=\sum_{r=b,\,r\,\mathrm{odd}}^{N-2}(-1)^{b+1}\binom{r-1}{b-1}\,f_rf_{N-r}+\sum_{r=a+1,r\,\mathrm{odd}}^{N-2}
        (-1)^{a}\binom{r-1}{a-1}\,f_rf_{N-r}$$ modulo $f_2^{\frac N 2}$. 
        As $a$ and $b$ are even, we obtain $$\phi(\zeta^\m(a,b))=\sum_{r=a+1,\,r\,\mathrm{odd}}^{N-2}\binom{r-1}{a-1}\,f_rf_{N-r}-\sum_{r=b,\,r\,\mathrm{odd}}^{N-2}\binom{r-1}{b-1}\,f_rf_{N-r}\quad\mathrm{modulo}\;f_2^{\frac N 2}$$  which can be rewritten as in the theorem.\\
        
        \item Using \Cref{nmdiffparities}, we have: \begin{multline*}
\phi(\zeta^{\mathfrak m}(a,b))=\sum_{r=b,\,r\,\mathrm{odd}}^{N-2}\left((-1)^{b+1}\binom{r-1}{b-1}\,f_rf_{N-r}\right)+\sum_{r=a+1,\,r\,\mathrm{odd}}^{N-2}\left((-1)^{a}\binom{r-1}{a-1}\,f_rf_{N-r}\right)+c_{N}^\phi(\zeta^\m(a,b))f_N\\=\sum_{r=3,\,r\,\mathrm{odd}}^{N-2}\left(\binom{r-1}{b-1}+\binom{r-1}{a-1}\,\right)f_rf_{N-r} +c_{N}^\phi(\zeta^\m(a,b))f_N.
    \end{multline*} 
    Using the fact that $\phi$ is normalized (\Cref{phis}), we have: 
    \begin{equation}\label{expr1}
\zeta^{\mathfrak m}(a,b)=\sum_{r=3,\,r\,\mathrm{odd}}^{N-2}\left(\binom{r-1}{b-1}+\binom{r-1}{a-1}\,\right)\zeta^\m(r)\zeta^\m(N-r) +c_{N}^\phi(\zeta^\m(a,b))\zeta^\m(N)
    \end{equation} where we have used the fact that in this case the quantity $N-r$ is even.
    We apply the period map $\mathrm{per}$ from \Cref{permap} to Equation \eqref{expr1} and we obtain the equality of complex numbers \eqref{expr11}.
    \begin{equation}\label{expr11}
   \zeta(a,b)=\sum_{r=3,\,r\,\mathrm{odd}}^{N-2}\left(\binom{r-1}{b-1}+\binom{r-1}{a-1}\,\right)\zeta(r)\zeta(N-r) +c_{N}^\phi(\zeta^\m(a,b))\zeta(N).
    \end{equation}
    Now, we can compute $c_N^\phi$ as follows:
  $$c_{N}^\phi(\zeta^\m(a,b))=\frac{\zeta(a,b)-\sum_{r=3,\,r\,\mathrm{odd}}^{N-2}\left(\binom{r-1}{b-1}+\binom{r-1}{a-1}\,\right)\zeta(r)\zeta(N-r)}{\zeta(N)}\;\in\Q.$$ But in \cite[Section 5]{Borwein}, it is proved that: $$\zeta(a,b)=\sum_{r=3,\,r\,\mathrm{odd}}^{N-2}\left(\binom{r-1}{b-1}+\binom{r-1}{a-1}\,\right)\zeta(r)\zeta(N-r)-\frac 1 2\left(\binom N a+1\right)\zeta(N).$$ Hence, we obtain $$c_N^\phi=-\frac 1 2\left(\binom N a+1\right)$$ and it concludes the proof.
    \item This proof is the same as the one of \Cref{ecrfalpha}\ref{neven} but adapted to $a$ odd and $b$ even. 
 Using \Cref{nmdiffparities}, we have: \begin{multline*}
\zeta^{\mathfrak m}(a,b)=\sum_{r=b,\,r\,\mathrm{odd}}^{N-2}\left((-1)^{b+1}\binom{r-1}{b-1}\,f_rf_{N-r}\right)+\sum_{r=a+1,\,r\,\mathrm{odd}}^{N-2}\left((-1)^{a}\binom{r-1}{a-1}\,f_rf_{N-r}\right)+c_{N}^\phi(\zeta^\m(a,b))f_N\\=-\sum_{r=3,\,r\,\mathrm{odd}}^{N-2}\left(\binom{r-1}{b-1}+\binom{r-1}{a-1}\,\right)f_rf_{N-r} +\mathds 1_{\{a\geqslant 3\}}\,f_af_b+c_{N}^\phi(\zeta^\m(a,b))f_N.
    \end{multline*} We add the term $f_af_b$ because $a$ is odd and hence the variable $r$ in the sum can take the value $a$.
    \begin{multline*}
\zeta^{\mathfrak m}(a,b)=\sum_{r=b,\,r\,\mathrm{odd}}^{N-2}\left((-1)^{b+1}\binom{r-1}{b-1}\,f_rf_{N-r}\right)+\sum_{r=a+1,\,r\,\mathrm{odd}}^{N-2}\left((-1)^{a}\binom{r-1}{a-1}\,f_rf_{N-r}\right)+c_{N}^\phi(\zeta^\m(a,b))f_N\\=-\sum_{r=3,\,r\,\mathrm{odd}}^{N-2}\left(\binom{r-1}{b-1}+\binom{r-1}{a-1}\,\right)f_rf_{N-r} +\mathds 1_{\{a\geqslant 3\}}\,f_af_b+c_{N}^\phi(\zeta^\m(a,b))f_N.
    \end{multline*}
    Since the morphism $\phi$ is normalized (\Cref{phis}), we have:
    \begin{equation}\label{expr2}
\zeta^{\mathfrak m}(a,b)=-\sum_{r=3,\,r\,\mathrm{odd}}^{N-2}\left(\binom{r-1}{b-1}+\binom{r-1}{a-1}\,\right)\zeta^\m(r)\zeta^\m(N-r) +\mathds 1_{\{a\geqslant 3\}}\,\zeta^\m(a)\zeta^\m(b)+c_{N}^\phi(\zeta^\m(a,b))\zeta^\m(N)
    \end{equation} where we have used that the quantity $N-r$ is even.
   We apply the period map $\mathrm{per}$ from \Cref{permap} to Equation \eqref{expr2} and we obtain the equality of complex numbers \eqref{expr22}. \begin{equation}\label{expr22}\zeta(a,b)=-\sum_{r=3,\,r\,\mathrm{odd}}^{N-2}\left(\binom{r-1}{b-1}+\binom{r-1}{a-1}\,\right)\zeta(r)\zeta(N-r) +\mathds 1_{\{a\geqslant 3\}}\,\zeta(a)\zeta(b)+c_{N}^\phi(\zeta(a,b))\zeta(N).
   \end{equation} 
    Hence, we can compute $c_N^\phi$ as $$c_{N}^\phi(\zeta^\m(a,b))=\frac{\zeta(a,b)-\left(-\sum_{r=3,\,r\,\mathrm{odd}}^{N-2}\left(\binom{r-1}{b-1}+\binom{r-1}{a-1}\,\right)\zeta(r)\zeta(N-r)+\mathds 1_{\{a\geqslant 3\}}\,\zeta(a)\zeta(b)\right)}{\zeta(N)}\;\in\Q.$$ But in \cite[Section 5]{Borwein}, it is proved that: $$\zeta(a,b)=-\sum_{r=3,\,r\,\mathrm{odd}}^{N-2}\left(\binom{r-1}{b-1}+\binom{r-1}{a-1}\,\right)\zeta(r)\zeta(N-r)+\frac 1 2\left(\binom N b-1\right)\zeta(N)+\mathds 1_{\{a\geqslant 3\}}\,\zeta(a)\zeta(b).$$ Hence, we obtain $$c_N^\phi=\frac 1 2\left(\binom N b-1\right)=\frac 1 2\left(\binom N a-1\right)$$ and it concludes the proof.\qedhere
    \end{enumerate}\end{proof}
   
    \begin{rem}\label{borw}
        Let us give a comment on the strategy in the proof of \Cref{ecrfalpha}\ref{neven} and \ref{nodd}. Looking at \Cref{nmdiffparities}, we still miss a term given by $(c_{N}\circ\phi)(\zeta^m(a,b))$. To compute it, we apply the period map and we use computations in real numbers from \cite{Borwein} to find this missing term. This strategy is similar to the one used by Francis Brown in \cite[Theorem 4.1]{Bro12a}, allowing him to lift to the motivic context a formula on multiple zeta values due to Zagier \cite{zagier}.
    \end{rem}
\begin{ex}\label{gdex}
\begin{enumerate}
    \item 
\label{fnfn}
    Let $a\geqslant 3$ be an odd integer. We have $$\phi(\zeta^{\mathfrak m}(a,a))=f_af_a \quad\mathrm{modulo}\;f_2^a$$ as a consequence of \Cref{ecrfalpha}\ref{writingfalphanmodd}.

\item\label{zeta37}
Let us compute $\phi(\zeta^\m(3,7))$. We obtain $$\phi(\zeta^\m(3,7))=\left(\binom 6 6-\binom 6 4\right)f_7f_3-\binom4 2f_5f_5 =-14f_7f_3-6f_5f_5\quad\mathrm{modulo}\;f_2^5$$ as a consequence of \Cref{ecrfalpha}\ref{writingfalphanmodd}.

\item
    Let us compute $\phi(\zeta^\m(4,8))$ by \Cref{ecrfalpha}\ref{writingfalphanmeven}. We obtain $$\phi(\zeta^\m(4,8))=\binom 4 3f_5f_7+\binom 6 3 f_7 f_5+ \left(\binom 8 3-\binom 8 7 \right)f_9f_3=4f_5f_7+20 f_7f_5+48f_9f_3\quad\mathrm{modulo}\;f_2^6$$ as a consequence of \Cref{ecrfalpha}\ref{writingfalphanmeven}.

\item
    We have  $$\phi(\zeta^\m(4,3))=\left(\binom 4 2 +\binom 4 3\right)f_5f_2+f_3f_4-\frac 1 2\left(\binom 7 4 +1\right)f_7=10f_5f_2+\frac 2 5f_3f_2^2-18f_7$$ as a consequence of \Cref{ecrfalpha}\ref{neven}.

\item\label{zeta34}
    We have $$\phi(\zeta^\m(3,4))=-\left(\binom 4 3 +\binom 4 2 \right)f_5f_2+\frac 1 2\left(\binom 7 3 -1\right)=17f_7-10f_5f_2$$ as a consequence of \Cref{ecrfalpha}\ref{nodd}.
    
    \end{enumerate}
\end{ex}
\begin{rem}\label{perconjbb}
    Let us apply the period map to \Cref{ecrfalpha} \ref{neven} and \ref{nodd}. We obtain equalities in real numbers: 
    $$\zeta(a,b)=\sum_{r=3,\,r\,\mathrm{odd}}^{N-2}\left(\binom{r-1}{b-1}+\binom{r-1}{a-1}\,\right)\zeta(r)\zeta(N-r)-\frac 1 2\left(\binom N a+1\right)\zeta(N)$$ if $a$ is odd and $b$ even and 
    $$\zeta(a,b)=-\sum_{r=3,\,r\,\mathrm{odd}}^{N-2}\left(\binom{r-1}{b-1}+\binom{r-1}{a-1}\,\right)\zeta(r)\zeta(N-r)+\frac 1 2\left(\binom N b-1\right)\zeta(N)+\zeta(a)\zeta(b)$$ if $a$ is even and $b$ is odd. They were already known \cite[Section 5]{Borwein}. As explained in \Cref{borw}, these equalities are used in the proof. Now, we have proved that these equalities  have a motivic origin. This is in line with the period conjecture, which predicts that all linear relations between periods have a motivic origin.
\end{rem}
\begin{rem}
The decomposition of double zeta values in \Cref{ecrfalpha} implies the double-shuffle relations (\Cref{doubleshuffle}) in motivic double zeta values. These relations were already known to be motivic since the article \cite{souderes}. 
\end{rem}
\begin{rem}
In \cite{GKZ}, the authors use the double-shuffle relations (\Cref{doubleshuffle}) to prove that for $a$ and $b$ odd, a double zeta value $\zeta(a,b)$ verifies at least $\dim(S_{a+b})$ linearly independent relations, where $S_{a+b}$ is the space of cusp forms of weight $a+b$ in $\Gamma_1$. Since the double-shuffle relations are motivic (\cite{souderes}), the same result is true for motivic double zeta values.
\end{rem}

\section{Minimal motive and motivic Galois group of double zeta values}\label{subsecttangroup}
In this section, we compute the orbit of $\zeta^\m(a,b)$ under the action of $G_{\dR}$. From this, we deduce a description of the minimal motive $M(a,b)$ for $\zeta^\m(a,b)$. We then compute the Tannakian group $G(a,b)$ of $M(a,b)$. In particular, we obtain its dimension and its weight filtration.\\
Moreover, we introduce the notations $I(a,b)$, $d(a,b)$ and $J(a,b)$ in Definitions \ref{inm} and \ref{jnm} and they are used throughout this section. We use the normalized $\phi$ that was introduced in \Cref{phis}.

    \begin{defi}\label{inm}
        Let $a,b$ be positive integers with $b\geqslant 2$. We denote by $N$ the total weight $a+b$. Let us introduce the set $$I(a,b):=\left\{\begin{array}{ll}
        \{0,a,2a\}&\mathrm{if}\,a=b\\
    (\{ b,\dots, N-2 \}\cap (2\N+1))\cup\{0,N\}&\mathrm{if}\,b<a \\
    (\{ a+1,\dots,N-2\}\cap  (2\N+1))\cup\{0,N\}&\mathrm{if}\,b>a.
    \end{array}
    \right.$$ Let us define $d(a,b)$ to be the cardinality of $I(a,b)$.
    \end{defi}
    \begin{defi}\label{cardi} We keep notations of \Cref{inm}.
     Let us enumerate elements of $I(a,b)$ as  $$I(a,b)=\{i_1,i_2\dots,i_{d(a,b)-1},i_{d(a,b)}\}$$  with $$0= i_1<i_2<\dots<i_{d(a,b)-1}<i_{d(a,b)}=N.$$
    \end{defi}
    \begin{lemma} \label{d(ab)} We give the value of $d(a,b)$ introduced in \Cref{inm}. 
    \begin{enumerate}
    \item If $a=b$, the value of $d(a,b)$ is  $3$.
        \item If $b<a$ and $a$ is odd, the value of $d(a,b)$ is $\frac{a+3}{2}$.
        \item If $b<a$, $b$ is even and a is even, the value of $d(a,b)$ is $\frac{a}{2}+1$.
        \item If $b<a$, $b$ is odd and $a$ is even, the value of $d(a,b)$ is $\frac{a}{2}+2$.
    \item If $b>a$ and $b$ is even, the value of $d(a,b)$ is $\frac{b}{2}+1$.
     \item If $b>a$, $b$ is odd and $a$ is odd, the value of $d(a,b)$ is $\frac{b+1}{2}$.
     \item If $b>a$, $b$ is odd and $a$ is even, the value of $d(a,b)$ is $\frac{b+3}{2}$.
    \end{enumerate}
   
    \end{lemma}
    
    \begin{defi}\label{jnm} We keep notations of \Cref{inm}.
    
    We also define $$J(a,b):=\{N-j\,\vert\, i\in I(a,b)\}.$$
    We remark that the cardinality of $J(a,b)$ is $d(a,b)$. Let us enumerate elements of $J(a,b)$ as $$J(a,b)=\{0=j_{d(a,b)},j_{d(a,b)-1},\dots,j_2,j_1=N\}$$ such that $j_k=N-i_k$, where $i_k$ is a piece of notation introduced in \Cref{cardi}.
    \end{defi}
\begin{lemma}\label{IJdiff}
Let us keep notations of \Cref{inm} and let $a$ and $b$ be different integers with the same parity. Then the set $(I(a,b)\cap J(a,b))\setminus\{0,N\}$ is empty if and only $N$ is divisible by $4$ and $a= \frac N 2 -1$.
\end{lemma}    
\begin{proof}\begin{enumerate}
   \item If $N\equiv 2[4]$, then $\frac N 2$ is in $(I(a,b)\cap J(a,b))\setminus\{0,N\}$.
   \item If $4\mid N$ and $a\neq \frac N 2-1$, then $\frac N 2 +1$ is in $(I(a,b)\cap J(a,b))\setminus\{0,N\}$.\\
   
   If $a=\frac N 2 -1$ then the set $(I(a,b)\cap J(a,b))\setminus\{0,N\}$ is clearly empty.\qedhere
   \end{enumerate}
\end{proof}

  \begin{defi}\label{aij} Let $a$, $b$ be positive integers with $b\geqslant 2$ and $N$ be $a+b$. Let us define rational numbers $c_{i}(a,b)$ as follows:
    \begin{enumerate}
    \item  Let us suppose that $a$ and $b$ have the same parity. For all odd integers $i$ in $I(a,b)\setminus\{0,N\}$, we define: $$c_{i}(a,b)=
    (-1)^a\left(\binom{i-1}{a-1}-\binom{i-1}{b-1}\right)+\mathds 1_{\{i=a\}}
    $$ and if $i$ is an integer that is not in $I(a,b)\setminus\{0,N\}$ and that is not $2$, let us define $c_i(a,b)=0$. We define $c_2(a,b)$ to be the coefficient in front of $f_2^{\frac{a+b} 2}$ in $\phi(\zeta^\m(a,b))$.
    \item Let us suppose that $a$ and $b$ have a different parity. We define for $i$ in $I(a,b)\setminus \{0\}$: $$c_{i}(a,b)=\left\{\begin{array}{ll}
    -\frac 1 2\left(\binom{N}{a}+1\right)&\mathrm{if}\;i=N\;\mathrm{and}\;a\;\mathrm{is\;even}\\
    \frac 1 2\left(\binom N a-1\right)&\mathrm{if}\;\mathrm{if}\;i=N\;\mathrm{and}\;a\;\mathrm{is\;odd}\\
    (-1)^a\left(\binom{i-1}{a-1}+\binom{i-1}{b-1}\right)&\mathrm{if}\;i\neq N
    \end{array}
    \right.$$ and if $i$ is an integer which is not in $I(a,b)\setminus\{0\}$, let us define $c_i(a,b)=0$.
    \end{enumerate}
    \end{defi}
    \begin{rem}\label{aijutile} 
         Using \Cref{ecrfalpha}, we have \begin{multline*}
        \phi(\zeta^\m(a,b))=\sum_{i\in I(a,b)\setminus\{0,N\}}(c_i(a,b)f_if_{N-i})+c_N(a,b)f_N+c_2(a,b)f_2^{\frac{a+b} 2}\\=\sum_{k=2}^{d(a,b)-1}(c_{i_k}(a,b)f_{i_k}f_{j_k})+c_N(a,b)f_N+c_2(a,b)f_2^{\frac{a+b} 2}
         \end{multline*}
         where notations $i_k$ and $j_k$ are introduced in Definitions \ref{inm}, \ref{cardi} and \ref{jnm} and the $c_i(a,b)$ are the constants introduced in \Cref{aij}.
    \end{rem}

  \begin{thm}\label{orbit1} Let $a$, $b$ be different positive integers with $b\geqslant 2$. Let us denote by $N$ the total weight $a+b$.
    \begin{enumerate}
    \item\label{orbit11} Let $a$ be even. Then the orbit of $\zeta^\m(a,a)$ under the action of $U_{\dR}$ is $\{\zeta^\m(a,a)\}$.
    \item\label{orbit12} Let $a$ be odd. The orbit of $\zeta^\m(a,a)$ under the action of $U_{\dR}$ is $$U_{\dR}\cdot\zeta^\m(a,a)=\left\{\zeta^\m(a,a)+\lambda\zeta^\m(a)+\frac{\lambda^2} 2 \right\}_{\lambda\in \Q}$$
        \item\label{cas2} Let us suppose that $a$ and $b$ have same parity. The orbit of $\zeta^\m(a,b)$ under the action of $U_{\dR}$ is $$U_{\dR}\cdot\zeta^\m(a,b)=\zeta^\m(a,b)+\left(\sum_{i\in I(a,b)\setminus\{0,N\}}\Q \zeta^\m(i)\right)+\Q$$
        
        \item\label{orbit14} Let us suppose that $a$ and $b$ have different parities. The orbit of $\zeta^\m(a,b)$ under the action of $U_{\dR}$ is $$U_{\dR}\cdot\zeta^\m(a,b)=\zeta^\m(a,b)+\left(\sum_{i\in I(a,b)\setminus \{0,N\}}\Q \zeta^\m(N-i)\right)+\Q$$
    \end{enumerate}
        
    \end{thm}
    \begin{proof}

    \begin{enumerate}
\item  Let us compute the orbit of $\phi(\zeta^\m(a,a))$. Using \Cref{ecrfalpha}\ref{writingfalphanmeven}, the motivic period $\phi(\zeta^\m(a,a))$ is a power of $f_2$. Hence, the action of $U_{\dR}$ is trivial on $\phi(\zeta^\m(a,b))$. We apply $\phi^{-1}$ to obtain the result.
           \item Let us compute the action of $U_{\dR}$ on $\phi(\zeta^\m(a,b))$. Using \Cref{ecrfalpha}\ref{writingfalphanmodd}, we have $$\phi(\zeta^\m(a,a))=f_a f_a+c_2(a,a)f_2^{a}$$ where the coaction on $f_2$ is trivial. The coaction of $\mathcal U'$ (\Cref{shuffle}) on $f_af_a$ is given by $$\Delta f_a f_a=1\otimes f_af_a+ f_a\otimes f_a+f_af_a\otimes 1 .$$
Hence, for $g$ in $U_{\dR}$, the action of $g$ on $f_af_a$ is given by $$g\cdot f_a f_a=f_af_a +f_a\times f_a(g)+(f_af_a)(g)$$ where the words in the $f_i$'s are viewed as functions on $U_{\dR}$. Since the $f_i'$ are coordinates functions (\Cref{coord}), when $g$ describes all $U_{\dR}$, the number $f_a(g)$ describes all $\Q$. We have $f_af_a=\frac 1 2 f_a\shuffle f_a$. Hence, the number $(f_a f_a)(g)$ equals $\frac 1 2 (f_a(g))^2$. Knowing that $U_{\dR}$ acts trivially on $f_2$, we obtain $$U_{\dR}\cdot \phi(\zeta^\m(a,a))=\{f_af_a+c_2(a,a)f_2^{a}+\lambda f_a+\frac 1 2 \lambda^2\}_{\lambda\in \Q}.$$ We apply $\phi^{-1}$ to obtain the result.
   \item  
      Let us compute the action of $g$ on $\phi(\zeta^\m(a,b))$.  We recall that we have introduced the piece of notation $\zeta^{\mathfrak a}$ in \Cref{notalevels}. By abuse of notation, we denote by $\phi$ the normalised isomorphism from $\mathcal A^{\MT}$ to $\mathcal U'$ induced by the morphism $\phi$ from \Cref{phis}.\\
    Using \Cref{ecrfalpha}\ref{writingfalphanmodd},\ref{writingfalphanmeven}, we have $$\phi(\zeta^\m(a,b))=\sum_{i\in I(a,b)\setminus\{0,N\}}(c_i(a,b)f_if_{N-i})+c_2(a,b)f_2^{\frac{a+b} 2}=\sum_{k=2}^{d(a,b)-1}(c_{i_k}(a,b)f_{i_k}f_{j_k})+c_2(a,b)f_2^{\frac{a+b} 2}$$ where notations $i_k$ and $j_k$ are introduced in Definitions \ref{inm}, \ref{cardi} and \ref{jnm} and the $c_i(a,b)$ are constants introduced in \Cref{aij}. The coaction of $\mathcal U'$ on $\phi(\zeta^\m(a,b))$ is given by: $$\Delta \phi(\zeta^\m(a,b))=\phi(\zeta^\m(a,b))\otimes 1+\sum_{i\in I(a,b)\setminus\{0,N\}}c_i(a,b)f_i\otimes f_{N-i} + 1\otimes \phi(\zeta^{\mathfrak a}(a,b))$$ since the coaction is given by the deconcatenation (\Cref{shuffle}). Hence for $g$ an element of $U_{\dR}$, the action of $g$ on $\phi(\zeta^\m(a,b))$ is given by $$g\cdot \phi(\zeta^\m(a,b))=\phi(\zeta^\m(a,b))+\sum_{i\in I(a,b)\setminus\{0,N\}}c_i(a,b)f_i\times  f_{N-i}(g) + \phi(\zeta^{\mathfrak a}(a,b))(g)$$ where elements of $\mathcal U'$ are viewed as functions on $U_{\dR}$. Let us recall that the $f_i$'s and the words $f_if_j$ for $i<j$ are coordinate functions (\Cref{coord}). For $g$ in $U_{\dR}$, we have: 
   \begin{multline*}
    g\cdot \phi(\zeta^\m(a,b))=\phi(\zeta^\m(a,b))+\sum_{i\in I(a,b)\setminus\{0,N\}} c_i(a,b) f_{N-i}(g) f_i+\left(\sum_{i\in I(a,b)\cap\{3,\dots,\frac N 2-1\}} c_i(a,b) (f_if_{N-i})(g)\right)\\+\left(\sum_{i\in I(a,b)\cap\{\frac N 2 +1,\dots, N-2\}}c_i(a,b)((f_i\shuffle f_{N-i})(g)-(f_{N-i}f_i)(g))\right)+\mathds{1}_{\{N\equiv 2[4]\}}\frac{\left(f_{\frac N 2}\shuffle f_{\frac N 2}\right)(g)}{2} \end{multline*} where the last term appears only when $N$ is not divisible by $4$.  Let us rewrite the orbit as follows: 
    \begin{multline}\label{equaIJ}
    g\cdot \phi(\zeta^\m(a,b))=\phi(\zeta^\m(a,b))+\sum_{i\in I(a,b)\setminus\{0,N\}} c_i(a,b) f_{N-i}(g) f_i+\left(\sum_{i\in I(a,b)\cap\{3,\dots,\frac N 2-1\}} c_i(a,b) (f_if_{N-i})(g)\right)\\+\left(\sum_{i\in I(a,b)\cap\{\frac N 2 +1,\dots, N-2\}}c_{i}(a,b) (f_{i}\shuffle f_{N-i})(g)\right)-\left(\sum_{j\in J(a,b)\cap\{2,\dots,\frac N 2 -1\}}c_{N-j}(a,b)(f_{j}f_{N-j})(g)\right)\\+\mathds{1}_{\{N\equiv 2[4]\}}\frac{\left(f_{\frac N 2}\shuffle f_{\frac N 2}\right)(g)}{2} \end{multline}
    To simplify notations in this proof, let us denote the set $I(a,b)\cap \{3,\dots,\frac N 2-1\}$ (respectively $J(a,b)\cap \{2,\dots,\frac N 2-1\}$) by $I_-$ (respectively $J_-$). Let us also denote $I(a,b)\cap \{\frac N 2+1,\dots, N-2\}$ (respectively $J(a,b)\cap \{\frac N 2+1,\dots, N-2\}$) by $I_+$ (respectively $J_+$).\\
    
    Now let us suppose that $I(a,b)$ and $J(a,b)$ are different. Then \eqref{equaIJ} can be expressed as: \begin{multline*}
    g\cdot \phi(\zeta^\m(a,b))=\phi(\zeta^\m(a,b))+\sum_{i\in I(a,b)\setminus\{0,N\}} c_i(a,b) \lambda_{N-i} f_i+\left(\sum_{i\in I_+} c_i(a,b) \lambda_i\lambda_{N-i}\right)\\+\left(\sum_{i\in I_-}c_i(a,b)\lambda_{i,N-i}\right)-\left(\sum_{j\in J_-}c_{N-j}(a,b)\lambda_{j,N-j}\right)+\mathds{1}_{\{N\equiv 2[4]\}}\frac{\lambda_{\frac N 2}^2}{2} \end{multline*} where $\lambda_{i_1,\dots,i_r}$ is the evaluation at $g$ of the coordinate function associated to a Lyndon word $f_{i_1}\dots f_{i_r}$. Since we have supposed that $I(a,b)$ is different than $J(a,b)$, the sets $I_-$ and $J_-$ are different and hence the quantity $$\left(\sum_{i\in I_+} c_i(a,b) \lambda_i\lambda_{N-i}\right)\\+\left(\sum_{i\in I_-}c_i(a,b)\lambda_{i,N-i}\right)-\left(\sum_{j\in J_-}c_{N-j}(a,b)\lambda_{j,N-j}\right)+\mathds{1}_{\{N\equiv 2[4]\}}\frac{\lambda_{\frac N 2}^2}{2}$$ describes all $\Q$ and is independent of the other coordinates.\\
    
    Now, let us suppose that $I(a,b)$ equals $J(a,b)$. Then, we have $I_-=J_-$. The equation \eqref{equaIJ} can be expressed as:
    \begin{multline*}
    g\cdot \phi(\zeta^\m(a,b))=\phi(\zeta^\m(a,b))+\sum_{i\in I(a,b)\setminus\{0,N\}} c_i(a,b) \lambda_{N-i} f_i+\left(\sum_{i\in I_+} c_i(a,b) \lambda_i\lambda_{N-i}\right)\\+\left(\sum_{i\in I_-}(c_i(a,b)-c_{N-i}(a,b))\lambda_{i,N-i}\right)+\mathds{1}_{\{N\equiv 2[4]\}}\frac{\lambda_{\frac N 2}^2}{2} \end{multline*} where $\lambda_{i_1,\dots,i_r}$ is the evaluation at $g$ of the coordinate function associated to a Lyndon word $f_{i_1}\dots f_{i_r}$. Now let us remark that when $a$ and $b$ are different, there exists $i$ in $I_-$ such that the quantity $c_i(a,b)-c_{N-i}(a,b)$ is non zero. Hence the quantity $$\left(\sum_{i\in I_+} c_i(a,b) \lambda_i\lambda_{N-i}\right)+\left(\sum_{i\in I_-}(c_i(a,b)-c_{N-i}(a,b))\lambda_{i,N-i}\right)+\mathds{1}_{\{N\equiv 2[4]\}}\frac{\lambda_{\frac N 2}^2}{2}$$ describes all $\Q$ and is independent of the other coordinates.\\
    We obtain the following orbit for $\phi(\zeta^\m(a,b))$ under the action of $U_{\dR}$:  $$U_{\dR}\cdot\phi(\zeta^\m(a,b))=\phi(\zeta^\m(a,b))+\left(\sum_{i\in I(a,b)\setminus\{0,N\}}\Q f_i\right)+\Q.$$
    We obtain the final result by applying $\phi^{-1}$.
    
   \item Let us compute the orbit of $\phi(\zeta^\m(a,b))$. 
    Using \Cref{ecrfalpha}, we have $$\phi(\zeta^\m(a,b))=\sum_{i\in I(a,b)\setminus\{0,N\}}(c_i(a,b)f_if_{N-i})+c_N(a,b)f_N=\sum_{k=2}^{d(a,b)-1}(c_{i_k}(a,b)f_{i_k}f_{j_k})+c_N(a,b)f_N$$ where notations $i_k$ and $j_k$ are introduced in Definitions \ref{inm}, \ref{cardi} and \ref{jnm} and the $c_i(a,b)$ are non-zero constants depending on $a$ and $b$.\\
    Let us compute the coaction of $\mathcal U'$ on $\phi(\zeta^\m(a,b))$ as follows: $$\Delta\phi(\zeta^\m(a,b))=\phi(\zeta^\m(a,b))\otimes 1 + \sum_{i\in I(a,b)\setminus\{0,N\}}c_i(a,b) f_{N-i}\otimes f_i + c_N(a,b)\otimes f_N$$ because we have $\Delta f_2=f_2\otimes 1$. We have used the fact that the quantity $N-i$ is even for $i$ in the set $I(a,b)\setminus\{0,N\}$. 
    Hence for $g$ an element in $U_{\dR}$, the action of $g$ on $\phi(\zeta^\m(a,b))$ is $$g\cdot\phi(\zeta^\m(a,b))=\phi(\zeta^\m(a,b)) + \sum_{i\in I(a,b)\setminus\{0,N\}}c_i(a,b) f_{N-i}\times f_i(g)+c_N(a,b)f_N(g)$$  where elements of $\mathcal U'$ are viewed as functions on $U_{\dR}$.\\
    Since the $f_i$'s are coordinate functions on $U_{\dR}$ the $f_i(g)$'s describe all $\Q$ when $g$ describes $U_{\dR}$. Hence we obtain: $$U_{\dR}\cdot\phi(\zeta^\m(a,b))=\phi(\zeta^\m(a,b))+\left(\sum_{i\in I(a,b)\setminus\{0,N\}} \Q f_{N-i}\right)+\Q.$$
    We obtain the final result by applying $\phi^{-1}$.\qedhere
    \end{enumerate}
    \end{proof}
    \begin{cor}\label{coroorbit}Let $a$, $b$ be positive different integers with $b\geqslant 2$. Let us denote the total weight $a+b$ by $N$.
    \begin{enumerate}
    \item \label{aaeven}Let $a$ be even. Then the orbit of $\zeta^\m(a,a)$ under the action of $G_{\dR}$ is $\Q^*\zeta^\m(a,a)$.
    \item\label{aaodd} Let $a$ be odd. The orbit of $\zeta^\m(a,a)$ under the action of $G_{\dR}$ is $$G_{\dR}\cdot\zeta^\m(a,a)=\left\{t\zeta^\m(a,a)+\lambda\zeta^\m(a)+\frac{\lambda^2} 2 \right\}_{\lambda\in \Q,t\in\Q^*}$$
        
    \item \label{abdiff}Let $a$ and $b$ with the same parity. The orbit of $\zeta^\m(a,b)$ under the action of $G_{\dR}$ is $$G_{\dR}\cdot\zeta^\m(a,b)=\mathbb Q^*\zeta^\m(a,b)+\left(\sum_{i\in I(a,b)\setminus\{0,N\}}\Q \zeta^\m(i)\right)+\Q$$
   
    \item \label{absame}Let $a$ and $b$ with different parities. The orbit of $\zeta^\m(a,b)$ under the action of $G_{\dR}$ is $$\Q^*\zeta^\m(a,b)+\left(\sum_{i\in I(a,b)\setminus \{0,N\}}\Q \zeta^\m(N-i)\right)+\Q$$
    \end{enumerate}
    \end{cor}
    \begin{rem}
    In \Cref{underlyingvs},  we exhibit the underlying graded $\Q$-vector space $V(a,b)$ of the minimal motive $M(a,b)$ for a given motivic double zeta value $\zeta^\m(a,b)$, when viewed as a graded representation of $U_{\dR}$.
    \end{rem}
   \begin{cor}\label{underlyingvs}
   Let $a$, $b$ be different positive integers with $b\geqslant 2$.
   \begin{enumerate}
      \item\label{vs1}  Let $a$ be even. The underlying graded $\Q$-vector space $V(a,a)$ of $M(a,a)$ is $M_{2a}=\Q\zeta^\m(a,a)$.
       \item \label{vs2}Let $a$ be odd. The underlying $\Q$-vector space $V(a,a)$ of $M(a,a)$ is $$V(a,a)=M_0\oplus M_a\oplus M_{2a}$$ where $M_0$ is generated by $1$, $M_a$ is generated by $\zeta^\m(a)$ and $M_{2a}$ is generated by $\zeta^\m(a,a)$.
       \item \label{vs3} Let $a$ and $b$ with the same parity. The underlying graded $\Q$-vector-space $V(a,b)$ of $M(a,b)$ is $$V(a,b)=\bigoplus_{i\in I(a,b)} M_i$$ where $I(a,b)$ is introduced in \Cref{inm}. We have $M_0$ is generated by $1$, $M_N$ is generated by $\zeta^\m(a,b)$ and for all $i$ in $I(a,b)\setminus\{0,N\}$, $M_i$ is generated by $\zeta^\m(i)$.
       \item \label{vs4} Let $a$ and $b$ with different parities. The underlying graded $\Q$-vector space $V(a,b)$ of $M(a,b)$ is $$V(a,b)=\bigoplus_{j\in J(a,b)} M_j$$ where $J(a,b)$ is introduced in \Cref{jnm}. We have $M_0$ is generated by $1$, $M_N$ is generated by $\zeta^\m(a,b)$ and for all $j$ in $J(a,b)\setminus\{0,N\}$, $M_j$ is generated by $\zeta^\m(j)$.
   \end{enumerate}
    \end{cor}
\begin{proof} As a consequence of the Tannakian formalism (see for example \cite[Corollary 2.5]{brownnotesonmotper}), the minimal motive for a given motivic multiple zeta value $\zeta^\m$ is the representation of $G_{\dR}$ given by the $\Q$-vector space generated by the orbit $G_{\dR}\cdot\zeta^\m$. We compute the graded $\Q$-vector space generated by each orbit in \Cref{coroorbit}.
\end{proof}
    \begin{thm}\label{grosthmtangp} Let $a$, $b$ be different positive integers with $b\geqslant 2$. Let $N$ denote the total weight $a+b$. Let us give the Tannakian group $G(a,b)$ of $M(a,b)$.
        \begin{enumerate} 
       \item \label{pairpair}Let us suppose that $a$ is even. The group $G(a,a)$ is $\mathbb G_m$.
       \item \label{oddodd} Let $a$ be odd. Then the group $G(a,a)$ is the group scheme $$G(a,a)=\left\{\begin{pmatrix}
    1&\lambda&\frac{\lambda^2} 2\\
    0&t^{a}&\lambda\\
    0&0&t^{2a}
    \end{pmatrix}\right\}_{\lambda\in\Q}\subset GL_{3,\Q}$$
    \item \label{tanoddnm}
    Let us suppose that $a$ and $b$ are different integers with the same parity.
    Then, the group $G(a,b)$ is the group scheme $$\left\{\begin{pmatrix}
    1&\alpha_{i_2}&\alpha_{i_3}&\dots&\alpha_{i_{d(a,b)-1}}&\alpha_{i_{d(a,b)}}\\
    0&t^{i_2}&0&\dots &0&c_{i_{2}}(a,b)\alpha_{j_2}\\
    0&0&t^{i_3}&\dots&0&c_{i_{3}}(a,b)\alpha_{j_3}\\
    \vdots&\vdots&\vdots&&\vdots&\vdots\\
    0&0&0&\dots&t^{i_{d(a,b)-1}}&c_{i_{d(a,b)-1}}(a,b)\alpha_{j_{d(a,b)-1}}\\
    0&0&0&\dots&0&t^{N}
    \end{pmatrix}\right\}_{\alpha_i\in Q}\subset GL_{d(a,b),\Q}$$ where we have introduced the notations $i_k$ and $j_k$ in Definitions \ref{inm}, \ref{cardi} and \ref{jnm} and the constants $c_i(a,b)$ are introduced in \Cref{aij}.
   
 \item \label{tanevennm} Let us suppose that $a$ and $b$ have different parities.  The Tannakian group $G(a,b)$ of $M(a,b)$ is the group scheme $$\left\{\begin{pmatrix}
            1&0&0&\dots&0&*\\
    0&t^{j_2}&0&\dots &0&*\\
    0&0&t^{j_3}&\dots&0&*\\
    \vdots&\vdots&\vdots&&\vdots&\vdots\\
    0&0&0&\dots&t^{j_{d(a,b)-1}}&*\\
    0&0&0&\dots&0&t^{N}
        \end{pmatrix}\right\}\subset GL_{d(a,b),\mathbb Q}$$ where we have introduced the piece of notation $j_k$ in \Cref{jnm}.
        \end{enumerate}
    \end{thm}
   \begin{rem}
       The results of \Cref{underlyingvs}\ref{vs1} and  \Cref{grosthmtangp}\ref{pairpair} mean that the minimal motive for $\zeta(a,a)$ for $a$ even is $\Q(-2a)$. It is not surprising since in that case, the double zeta value $\zeta(a,a)$ is a rational multiple of $\pi^{2a}$.
   \end{rem}
   \begin{rem} We remark that when $I(a,b)\setminus\{0,N\}$ and $J(a,b)\setminus\{0,N\}$ are disjoint, there is no relations between the coefficients of $G(a,b)$. For $a$ and $b$ different of the same parity, these sets are disjoint if and only if $N$ is divisible by $4$ and $a=\frac N 2 -1$ as explained in \Cref{IJdiff}.
   
   \end{rem}
    \begin{proof} We recall that the sets $I(a,b)$ and $J(a,b)$ are introduced in Definitions \ref{inm}, \ref{cardi} and \ref{jnm}.
    \begin{enumerate} 
    \item This is just a consequence of the fact that $U_{\dR}$ acts trivially on $\zeta^\m(a,a)$ as stated in \Cref{orbit1}\ref{orbit11}.
    \item Let $a$ be odd. Using \Cref{underlyingvs}\ref{vs2}, the underlying graded $\Q$-vector space $V(a,a)$ of $M(a,a)$ has the basis $$\mathcal B=\{1,\zeta^\m(a),\zeta^\m(a,a)\}$$ which respects the graduation. Let us compute the action of $U_{\dR}$ on this basis. Let $g$ be an element of $U_{\dR}$. With the same computations as in the proof of \Cref{orbit1}\ref{orbit12}, the action of $g$ on $\phi(\zeta^\m(a,a))$ is $$g\cdot\phi(\zeta^\m(a,a))=\phi(\zeta^\m(a,a))+f_a(g)\times f_a+\frac{(f_a(g))^2}{2}$$ where the words in the $f_i$'s are viewed as functions of $U_{\dR}$. The action of $g$ on $\phi(\zeta^m(i))$ is given by $$g\cdot \phi(\zeta^\m(i))=f_a+f_a(g).$$ we apply $\phi^{-1}$ to obtain the action of $U_{\dR}$ on the elements of the basis $\mathcal B$ and we obtain the result by setting $\lambda:=f_a(g)$.

       \item Let us suppose that $a$ and $b$ have same parity. 
       Using \Cref{underlyingvs}\ref{vs3} the underlying graded $\Q$-vector space $V(a,b)$ of $M(a,b)$ has a basis $$\mathcal B=\{1,\zeta^\m(a,b)\}\cup\{\zeta^\m(i)\}_{i\in I(a,b)\setminus\{0,N\}}$$ which respects the graduation. Let us compute the action of $U_{\dR}$ on the elements of this basis. Let $g$ be an element of $U_{\dR}$. With the same computations as in the proof of \Cref{orbit1}\ref{cas2}, the action of $g$ on $\phi(\zeta^\m(a,b))$ is  \begin{equation}
    g\cdot \phi(\zeta^\m(a,b))=\phi(\zeta^\m(a,b))+\sum_{i\in I(a,b)\setminus\{0,N\}} c_i(a,b) f_{N-i}(g) f_i+\lambda_{d(a,b)} \end{equation} where $\lambda_{d(a,b)}$ is an independent coordinate. We also have $$g\cdot\phi(\zeta^\m(i))=f_i+f_i(g)$$ for $i$ in $I(a,b)\setminus\{0,N\}$. Hence we obtain the result with $\alpha_i:=f_i(g)$.

       \item Let us suppose that $a$ and $b$ have different parities. 
       Using \Cref{underlyingvs}\ref{vs4}, the underlying graded $\Q$-vector space $V(a,b)$ of $M(a,b)$ has a basis $$\mathcal B=\{1,\zeta^\m(a,b)\}\cup\{\zeta^\m(j)\}_{j\in J(a,b)\setminus\{0,N\}}$$ which respects the graduation.
       Let us compute the action of $U_{\dR}$ on the elements of this basis. Let $g$ be an element of $U_{\dR}$. With the same computations as in the proof of \Cref{orbit1}\ref{orbit14}, the action of $g$ on $\phi(\zeta^\m(a,b))$ is  $$g\cdot\phi(\zeta^\m(a,b))=\phi(\zeta^\m(a,b))+\sum_{j\in J(a,b)\setminus\{0,N\}}c_{N-j}(a,b)f_{N-j}(g)f_j+c_N(a,b)f_N$$ and we also have
       $$g\cdot\phi(\zeta^\m(j))=f_j$$ for $j$ in $J(a,b)\setminus\{0,N\}$.
      We obtain the announced result.\qedhere
      
         \end{enumerate}
    \end{proof}
    \begin{cor}\label{dimi}We recall that $d(a,b)$ is introduced in \Cref{cardi} and its value is defined in \Cref{d(ab)}. Let $a$, $b$ be different positive integers with $b\geqslant 2$.
    \begin{enumerate}
    \item Let $a$ be even. The group $G(a,a)$ has dimension $1$.
    \item Let $a$ be odd. The group $G(a,a)$ has dimension $2$.
    \item \label{dimi2} 
    Let $a$, $b$ have the same parity.  
    The dimension of the group $G(a,b)$ is $2\,d(a,b)-\vert I(a,b)\cap J(a,b)\vert.$  
 \item        Let $a,b$ have different parities. The Tannakian group $G(a,b)$ has dimension $d(a,b)$.
        \end{enumerate}
    \end{cor}
  
   \begin{cor}Let $a$, $b$ be different positive integers with $b\geqslant 2$ and let us denote by $N$ the total weight $a+b$. 
   \begin{enumerate}
       \item Let $a$ be even. The only non-zero weight-graded piece in the weight filtration of $M(a,a)$ is $$\Gr_{2a}^WM(a,a)=\Q(-2a).$$
       \item Let $a$ be odd. The only non-zero weight-graded pieces in the weight filtration of $M(a,a)$ are $$Gr_i^WM=\Q(-i)$$ for $i$ in $\{0,a,2a\}$.
       \item Let $a$ and $b$ with the same parity. The only non-zero weight-graded pieces in the weight filtration of $M(a,b)$ are $$Gr_i^WM=\Q(-i)$$ for $i$ in $ I(a,b)$, where $I(a,b)$ is introduced in \Cref{inm}.
       \item Let $a$ and $b$ with different parities. The only non-zero weight-graded pieces in the weight filtration of $M(a,b)$ are $$Gr_j^WM=\Q(-j)$$ for $j$ in $J(a,b)$, where $J(a,b)$ is introduced in \Cref{jnm}.
   \end{enumerate}
   \end{cor}
   
    \begin{ex}\label{M35}
        Let us recall that we have $$\phi(\zeta^\m(3,5))=-5f_5f_3$$ as shown in Examples \ref{zeta35d3} and \ref{zeta35d5}. Let us compute $G(3,5).$ We obtain $$\left\{\begin{pmatrix}
            1&*&*\\
            0&t^5&*\\
            0&0&t^8
        \end{pmatrix}\right\}$$ where $t$ describes $\Q$. In this case $I(3,5)=\{0,5,8\}$ and $J(3,5)=\{0,3,8\}$.
        The dimension of the group is $4$ which is exactly $2\,d(3,5)-2$.
    \end{ex}
    \begin{ex} \label{M37}In \Cref{gdex}\ref{zeta37}, we have obtained $\zeta^\m(3,7)=-6f_5f_5-14f_7f_3$. In this case, the set $I(3,7)$ is $\{0,5,7,10\}$ and $J(3,7)$ is $\{0,5,3,10\}$. We remark that $I(3,7)\cap J(3,7)$ has cardinality $3$. The value of $d(3,7)$ is $4$. The Tannakian group associated to it is $$G(3,7)=\left\{\begin{pmatrix}
        1&\alpha_5&\alpha_7&\alpha_{10}\\
        0&t^5&0&-3\alpha_5\\
        0&0&t^7&\alpha_3\\
        0&0&0&t^{10}
    \end{pmatrix}\right\}_{\alpha_i\in\Q}\subset GL_{4,\Q}$$  The dimension of the group is $5$ which is exactly $2\,d(3,7)-\vert I(3,7)\cap J(3,7)\vert$.
        
    \end{ex}

    \begin{ex}\label{M34}
        We recall (\Cref{gdex}\ref{zeta34}) that $$\zeta^\m(3,4)=-10f_5f_2+17f_7.$$ We have $I(a,b)=\{0,5,7\}$ and $J(a,b)=\{0,2,7\}$. The Tannakian group is $$\left\{\begin{pmatrix}
            1&0&*\\
            0&t^2&*\\
            0&0&t^7
        \end{pmatrix}
        \right\}$$ which has dimension $3=d(3,4)$.
    \end{ex}
   \section{Period matrix}\label{transsection}
    In this section, we compute a period matrix of $M(a,b)$ and we formulate a conjecture about algebraic relations between double zeta values. Moreover, we deduce some predictions concerning $M(a,b)$ by using the period conjecture, which predicts that the period map $\mathrm{per}$ from \Cref{permap} is injective. Indeed, as a consequence of the period conjecture the transcendence degree of the $\Q$-algebra generated by the coefficients of a period matrix of $M(a,b)$ is expected to be equal to the dimension of $G(a,b)$ (which is computed in \Cref{dimi}). 
    \begin{thm}\label{permat}Let $a$, $b$ be different positive integers with $b\geqslant 2$ and let us denote by $N$ the total weight $a+b$. Let $M(a,b)$ be the minimal motive for $\zeta^\m(a,b)$. 
    \begin{enumerate}
        \item Let $a$ be even. A period matrix of $M(a,a)$ is $P(a,a)=(\zeta(a,a))$.
        \item Let $a$ be odd. A period matrix of $M(a,a)$ is $$P(a,a)=\begin{pmatrix}
        1&\zeta(a)&\zeta(a,a)\\
        0&(2\pi i)^{a}&(2\pi i)^a\zeta(a)\\
        0&0&(2\pi i)^{2a}
        \end{pmatrix}\in GL_{d(a,a)}(\mathbb C)$$
        \item \label{permatbothodd} Let $a$ and $b$ with the same parity. A period matrix of $M(a,b)$ is 
        $$P(a,b)=\begin{pmatrix}
    1&\zeta(i_2)&\zeta(i_3)&\dots&\zeta(i_{d(a,b)-1})&\zeta(a,b)\\
    0&(2\pi i)^{i_2}&0&\dots &0&(2\pi i)^{i_2}\zeta(j_2)\\
    0&0&(2\pi i)^{i_3}&\dots&0&(2\pi i)^{i_3}\zeta(j_3)\\
    \vdots&\vdots&\vdots&&\vdots&\vdots\\
    0&0&0&\dots&(2\pi i) ^{i_{d(a,b)-1}}&(2\pi i)^{i_{d(a,b)-1}}\zeta(j_{d(a,b)-1})\\
    0&0&0&\dots&0&(2\pi i)^{N}
    \end{pmatrix}\in GL_{d(a,b)}(\mathbb C)$$ where notations for $i_k$ and $j_k$ are introduced in Definitions \ref{cardi}, \ref{jnm}.
    \item \label{permatbothodd2}Let $a$ and $b$ with different parities. A period matrix of $M(a,b)$ is  $$P(a,b)=\begin{pmatrix}
            1&0&0&\dots&0&\zeta(N)\\
    0&(2\pi i)^{j_{d(a,b)-1}}&0&\dots &0&(2\pi i)^{j_{d(a,b)-1}}\zeta(i_{d(a,b)-1})\\
    0&0&(2\pi i)^{j_{d(a,b)-2}}&\dots&0&(2\pi i)^{j_{d(a,b)-2}}\zeta(i_{d(a,b)-2})\\
    \vdots&\vdots&\vdots&&\vdots&\vdots\\
    0&0&0&\dots&(2\pi i)^{j_{2}}&(2\pi i)^{j_2}\zeta(i_2)\\
    0&0&0&\dots&0&(2\pi i)^{N}
        \end{pmatrix}\in GL_{d(a,b)}(\mathbb C)$$
        where notations for $i_k$ and $j_k$ are introduced in \Cref{cardi}, \ref{jnm}.
    \end{enumerate}
    \end{thm}
    \begin{rem}
    In \Cref{permat}\ref{permatbothodd2}, the number $\zeta(a,b)$ does not appear in the period matrix of $M(a,b)$ that we have computed but it lies in the $\Q$-vector space generated by the coefficients of this period matrix. Indeed, either  by using \cite[Section 5]{Borwein} or by applying the period map to \Cref{ecrfalpha}\ref{neven},\ref{nodd}, we obtain:
    $$\zeta(a,b)=\sum_{r=3,\,r\,\mathrm{odd}}^{N-2}\left(\binom{r-1}{b-1}+\binom{r-1}{a-1}\,\right)\zeta(r)\zeta(N-r)-\frac 1 2\left(\binom N a+1\right)\zeta(N)$$ if $a$ is even and $b$ is odd and 
    $$\zeta(a,b)=-\sum_{r=3,\,r\,\mathrm{odd}}^{N-2}\left(\binom{r-1}{b-1}+\binom{r-1}{a-1}\,\right)\zeta(r)\zeta(N-r)+\frac 1 2\left(\binom N b-1\right)\zeta(N)+\zeta(a)\zeta(b)$$ if $a$ is odd and $b$ is even (using that $\zeta(2n)$ is a rational multiple of $(2\pi i)^{2n}$ for all positive integers $n$).
    \end{rem}
    \begin{rem}
    In fact, in the proof, we even compute explicitly a matrix of motivic periods of $M(a,b)$. 
    \end{rem}
    \begin{rem}Let $a$ and $b\geqslant 2$ be different positive integers and let $N$ be $a+b$. Let us explain that most of of the entries in the matrix $P(a,b)$ come from sub-objects and sub-quotients of $M(a,b)$.
    \begin{enumerate}
    \item Let $a$ and $b$ with the same parity. Let us explain that for all $i$ in $I(a,b)\setminus\{0,N\}$, the minimal motive $M(\zeta^\m(i))$ for $\zeta^\m(i)$ is a sub-motive of $M(a,b)$.\\
    Indeed, using \Cref{coroorbit}\ref{cas2}, we see that the graded $\Q$-vector space generated by $G_{\dR}\cdot\zeta^\m(i) $ is a sub-vector space of the underlying vector space $V(a,b)$ of $M(a,b)$. Since the action of $G_{\dR}$ on $M(\zeta^\m(i))$ is obtained by restriction of the action of $G_{\dR}$ on $M(a,b)$, the $G_{\dR}$-module $M(\zeta^\m(i))$ is a sub-module of $M(a,b)$.\\ 
    As a consequence, all $\zeta(i)$ for $i$ in $I(a,b)\setminus\{0,N\}$ are periods of $M(a,b)$.
    
    \item Let $a$ and $b$ with the same parity. Let us explain that for all $j$ in $J(a,b)\setminus\{0,N\}$, the $(N-j)$-th twist $M(\zeta^\m(j))(N-j)$ of the minimal motive for $\zeta^\m(j)$ is a sub-quotient of $M(a,b)$.\\
    Indeed, using notations from \Cref{underlyingvs}\ref{vs3}, it is given by the graded $\Q$-vector space $$V(j)(N-j):=V(a,b)/(\bigoplus_{i\in I(a,b)\setminus\{N-j,N\}}M_i)$$ endowed with an action of $G_{\dR}$. The motive $M(\zeta^\m(j))(N-j)$ is the (conjectural) minimal motive for the period $(2\pi i)^{N-j}\zeta(j)$ which appears in the period matrix of \Cref{permat}\ref{permatbothodd}. Since the action of $G_{\dR}$ on $M(\zeta^\m(j))(N-j)$ is obtained by quotienting the action of $G_{\dR}$ on $M(a,b)$, the $G_{\dR}$-module $M(\zeta^\m(j))(N-j)$ is a sub-quotient of the $G_{\dR}$-module $M(a,b)$. \\
    As a consequence, all $(2\pi i)^{k}\zeta(N-k)$ for $k\in I(a,b)\setminus\{0,N\}$ are periods of $M(a,b)$. 
    \item Let $a$ and $b$ with different parities. We explain that for all $k$ in $I(a,b)\setminus\{0,N\}$, the $(N-k)$-th twist $M(\zeta^\m(k))(N-k)$ of the minimal motive for $\zeta^\m(k)$ is a sub-quotient of $M(a,b)$. \\
    Indeed, using notations from \Cref{underlyingvs}\ref{vs3}, it is given by the graded $\Q$-vector space $$V(k)(N-k):=V(a,b)/(\bigoplus_{j\in J(a,b)\setminus\{N-k,N\}}M_j)$$ endowed with an action of $G_{\dR}$. The motive $M(\zeta^\m(k))(N-k)$ is the (conjectural) minimal motive for the period $(2\pi i)^{N-k}\zeta(k)$ which appears in the period matrix of \Cref{permat}\ref{permatbothodd2}. Since the action of $G_{\dR}$ on $M(\zeta^\m(k))(N-k)$ is obtained by quotienting the action of $G_{\dR}$ on $M(a,b)$, the $G_{\dR}$-module $M(\zeta^\m(k))(N-k)$ is a sub-quotient of the $G_{\dR}$-module $M(a,b)$.\\
    As a consequence, all $(2\pi i)^{N-k}\zeta(k)$ for $k$ in $I(a,b)$ are periods of $M(a,b)$. 
    \end{enumerate}
    \end{rem}
    \begin{proof}[Proof of \Cref{permat}]
   
    In the proof of this theorem, we use motivic periods as introduced in \cite{brownnotesonmotper}. They are elements of $\mathcal O(G_{\dR})$. They are symbols $[M,v,f]$ where $M$ is in $\MT(\Z)$, $v$ is in $\omega_{\dR}(M)$ and $f$ is in $\omega_{\B}(M)^\vee$. They are endowed with the structure of an algebra \cite[Section 2.2]{brownnotesonmotper}.\\
    
   We consider the underlying graded vector space $V(a,b)$ of $M(a,b)$. A basis of $V(a,b)$ is exhibited in \Cref{underlyingvs}. It provides a basis $\mathcal B_{\dR}$ of $\omega_{\dR}(M)$ and a basis $\mathcal B_{\B}$ of $\omega_{\B}(M)^\vee$. Both these bases respect the graduation. The motivic period that we compute in these bases are the coordinate functions on the matrix of an element of $G(a,b)$. Hence, in the period matrix of $M(a,b)$ in these bases, the non-zero coefficients are given by the non-zero coefficients of $G(a,b)$ (where $G(a,b)$ is described in \Cref{grosthmtangp}). \\
    
   Let us compute the motivic periods of $M(a,b)$ and then their image under $\mathrm{per}$ (defined in \Cref{permap}).
   \begin{enumerate}
       \item Let $a$ be even. There is one non-zero coefficient in the matrices of the group $G(a,a)$. Hence, we are looking for one motivic period, which is $\zeta^\m(a,a)$ by construction. It is sent to $\zeta(a,a)$ by the period application $\mathrm{per}$ (\Cref{permap}).
       \item Let $a$ be odd. Let us denote the basis $\mathcal B_{\dR}$ of $V(a,b)$ by  $\{e_i\}_{i\in \{0,a,2a\}}$. Let us denote $\{g_i\}_{i\in \{0,a,2a\}}$ the basis $\mathcal B_{\B}$.\\
       There are $6$ non-zero coefficients in the matrices of the group $G(a,a)$. Hence the period matrix $P(a,a)$ of $M(a,a)$ in the bases $\mathcal B_{\dR}$ and $\mathcal B_{\B}$ has $6$ non-zero coefficients.\\
       Motivic periods in $P(a,a)$ are given by all $[M(a,a),e_i,g_j]$. \\Let $k$ be in $\{0,a,2a\}$. Then $\mathrm{per}([M(a,a),e_k,g_k])=\mathrm{per}([\Q(-k),e_k,g_k])=(2\pi i)^k$.\\
      Let us denote by $M(\zeta^\m(a))(a)$ the $a$-th twist of the minimal motive for $\zeta^\m(a)$. 
       We have $$[M(a,a),e_0,g_{a}]=\zeta^\m(a)\quad\mathrm{and}\quad[M(a,a),e_a,g_{2a}]=[M(\zeta^\m(a))(a),e_a,g_{2a}]$$ which are respectively sent to $\zeta(a)$ and to $(2\pi i)^a\zeta(a)$.\\
       By construction, the number $\zeta(a,a)$ is a period of $M(a,a)$. We have computed the $6$ coefficients of the period matrix $P(a,a)$ of $M(a,a)$ in the bases $\mathcal B_{\dR}$ and $\mathcal B_{\B}$.

       \item Let $a$ and $b$ with the same parity. 
       Let us denote the basis $\mathcal B_{\dR}$ of $V(a,b)$ as  $\{e_i\}_{i\in I(a,b)}$. Let us denote $\{g_i\}_{i\in I(a,b)}$ the basis $\mathcal B_{\B}$.\\
       There are $3d(a,b)-3$ non-zero coefficients in the matrices of the group $G(a,b)$. Hence the period matrix $P(a,b)$ of $M(a,b)$ in the bases $\mathcal B_{\dR}$ and $\mathcal B_{\B}$ has $3d(a,b)-3$ non-zero coefficients.\\
        Then, motivic periods of $M(a,b)$ are generated by all $[M(a,b),e_i,g_j]$.\\
       Let $k$ be in $I(a,b)$. Then $$\mathrm{per}([M(a,b),e_k,g_k])=\mathrm{per}([\Q(-k),e_k,g_k])=(2\pi i)^k.$$ This represents $d(a,b)$ non-zero coefficients in $P(a,b)$. \\
       Now let $k$ be in $I(a,b)\setminus\{0,N\}$ and let us denote by $M(\zeta^\m(k))$ the minimal motive for $\zeta^\m(k)$. Then $[M(a,b),v_0,g_k]=[M(\zeta^\m(k)),v_0,g_k]$ and hence $$\mathrm{per}([M(a,b),v_0,g_k])=\zeta(k).$$ This represents $d(a,b)-2$ periods.\\ For $k$ in $I(a,b)\setminus\{0,N\}$, we also have $[M(a,b),v_k,g_N]=[M(\zeta^\m(N-k))(k),v_k,g_N]$ where $M(\zeta^\m(N-k))(k)$ is the $k$-th twist of the minimal motive for $\zeta^\m(N-k)$. Hence $$\mathrm{per}([M(a,b),v_k,g_N])=(2\pi i)^k\zeta(N-k).$$ This represents $d(a,b)-2$ periods.\\
       By construction, $\zeta(a,b)$ is also a period of $M(a,b)$. We have computed the $3d(a,b)-3$ coefficients of the period matrix $P(a,b)$ of $M(a,b)$ in the bases $\mathcal B_{\dR}$ and $\mathcal B_{\B}$.

       \item Let $a$ and $b$ with different parities. 
       Let us denote the basis $\mathcal B_{\dR}$ of $V(a,b)$ as  $\{e_i\}_{i\in I(a,b)}$. Let us denote $\{g_i\}_{i\in I(a,b)}$ the basis $\mathcal B_{\B}$.\\
       There are $2d(a,b)-1$ non-zero coefficients in the matrices of the group $G(a,b)$. Hence the period matrix $P(a,b)$ of $M(a,b)$ in the bases $\mathcal B_{\dR}$ and $\mathcal B_{\B}$ has $2d(a,b)-1$ non-zero coefficients.\\
       Then, coefficients in $P(a,b)$ are given by all $[M(a,b),e_i,g_j]$.\\
       Let $k$ be in $I(a,b)$. Then $$\mathrm{per}([M(a,b),e_k,g_k])=\mathrm{per}([\Q(-k),e_k,g_k])=(2\pi i)^k.$$ This represents $d(a,b)$ periods. \\
       Now let $k$ be in $I(a,b)\setminus\{0\}.$ Then $[M(a,b),v_{N-k},g_N]=[M(\zeta(k))(N-k),v_{N-k},g_N]$ and hence $$\mathrm{per}([M(a,b),v_{N-k},g_N])=(2\pi i)^{N-k}\zeta(k)$$ where $M(\zeta^\m(N-k))(k)$ is the $k$-th twist of the minimal motive for $\zeta^\m(N-k)$. This represents $d(a,b)-1$ periods.\\ 
     We have computed the $2d(a,b)-1$ coefficients of the period matrix $P(a,b)$ of $M(a,b)$ in the bases $\mathcal B_{\dR}$ and $\mathcal B_{\B}$.\qedhere
       \end{enumerate}
    \end{proof}
    
   \begin{rem}\label{remz}\begin{enumerate}
       \item  For $n$ an integer, let us write $\widetilde{\mathcal Z}_n$ the vector space spanned by all numbers $(2\pi i)^k\zeta(n_1,\dots,n_r)$ such that $k+n_1+\dots+n_r=n$. Let us remark that $(2\pi i)\widetilde{\mathcal Z}_{n-1}\subset\widetilde{\mathcal Z}_n$. It corresponds to the vector space spanned by elements $(2\pi i)^k\zeta(n_1,\dots,n_r)$ of $\widetilde{\mathcal Z}_n$ where $k\geqslant 1$.\\
   
   Let $a$ and $b$ different positive integers having the same parity with $b\geqslant 2$. Let us denote by $N$ the total weight $a+b$. Let us consider $\widetilde{\mathcal Z}_N$. Using \Cref{permat}\ref{permatbothodd}, we remark $$\mathcal P(M(a,b))\cap \widetilde{\mathcal Z}_{N}/(2\pi i)\widetilde{\mathcal Z}_{N-1}=\Q\zeta(a,b)$$ where $\mathcal P(M(a,b))$ is the set of periods of $M(a,b)$.
In other words, $\zeta(a,b)$ is the only period of $M(a,b)$ which is not a period of weight lower than $N$ or a twist of such a period. 

\item\label{twists2} If $a$ and $b$ have different parities, the quantity $\zeta(a,b)$ is just a sum of twists of single zeta values. It can be seen either using \Cref{ecrfalpha}\ref{neven},\ref{nodd} or using the relations from \cite[Section 5]{Borwein} which are recalled in \Cref{perconjbb}.
\end{enumerate}
\end{rem}

\begin{ex}
Let us give a period matrix of $M(3,5)$ which is described in \Cref{M35}. We obtain $$P(3,5)=\begin{pmatrix}
            1&\zeta(5)&\zeta(3,5)\\
            0&(2\pi i)^5&(2\pi i)^5\zeta(3)\\
            0&0&(2\pi i)^8
        \end{pmatrix}\in GL_{3,\Q}(\mathbb C)$$
\end{ex}
\begin{ex}
Let us give a period matrix of $M(3,7)$ which is described in \Cref{M37}. We obtain

$$P(3,7)=\begin{pmatrix}
        1&\zeta(5)&\zeta(7)&\zeta(3,7)\\
        0&(2\pi i)^5&0&(2\pi i)^5\zeta(5)\\
        0&0&(2\pi i)^7&(2\pi i )^7\zeta(3)\\
        0&0&0&(2\pi i)^{10}
    \end{pmatrix}\in GL_{4,\Q}(\mathbb C)$$ 
\end{ex}
\begin{ex}
Let us consider the motive $M(3,4)$. Let us recall that the group $G(3,4)$ is described in \Cref{M34}. We obtain the period matrix
$$P(3,4)=\begin{pmatrix}
            1&0&\zeta(7)\\
            0&(2\pi i)^2&(2\pi i)^2\zeta(5)\\
            0&0&(2\pi i)^7
        \end{pmatrix}\in GL_{3,\Q}(\mathbb C)
        $$
        Using \cite[Section 5]{Borwein} or by applying the period map to the equality $$\zeta^\m(3,4)=17\zeta^\m(7)-10\zeta^\m(5)\zeta^\m(2)$$ from \Cref{gdex}\ref{zeta34}, 
         we have $$\zeta(3,4)=17\zeta(7)-10\zeta(5)\zeta(2)$$ and thus $\zeta(3,4)$ is a period of $M(3,4)$ even though it does not appear in $P(3,4)$.
\end{ex}

\begin{rem}
We know that the transcendence degree of the algebra generated by periods of $M(a,b)$ (computed in \Cref{permat}) is at most the dimension of $G(a,b)$ (computed in \Cref{dimi}). However this fact does not provide new inequalities. Let us develop this.
\begin{enumerate}
\item For $a$ odd, the group $G(a,a)$ is of dimension $2$ which implies that $\zeta(a,a)$ is $\Q$-algebraically dependant of $\pi$ and $\zeta(a)$ but this is already known thanks to the equality $$2\zeta(a,a)=\zeta(a)^2-\zeta(2a)$$ where $\zeta(2a)$ is a rational multiple of a power of $\pi$.
\item For $a$ and $b$ having the same parity, it implies that the transcendence degree of the algebra generated by $$\{(2\pi i)^k,\zeta(k),(2\pi i)^{k}\zeta(N-k),\zeta(a,b)\}_{k\in I(a,b)\setminus\{0,N\}}$$ is less than $2d(a,b)-\vert I(a,b)\cap J(a,b)\vert$. But there are exacly $2d(a,b)-\vert I(a,b)\cap J(a,b)\vert$ distinct elements in the set $$\{2\pi i,\zeta(k),\zeta(N-k),\zeta(a,b)\}_{k\in I(a,b)\setminus\{0,N\}}.$$ Hence the inequality is immediate.
\item For $a$ and $b$ having different parities, it implies that the transcendence degree of the algebra generated by $$\{(2\pi i)^{N-k}, (2\pi i)^{N-k}\zeta(k)\}_{k\in I(a,b)\setminus\{0\}}$$ is less than $d(a,b)$. But there are exactly $d(a,b)$ elements in the set $$\{2\pi i,\zeta(k)\}_{k\in I(a,b)\setminus\{0\}}$$ and hence the inequality is immediate.
\end{enumerate}
\end{rem}

\begin{conj} \label{corodimtrans}Let $a$, $b$ be different positive integers having the same parity with $b\geqslant 2$. The transcendence degree of the $\Q$-extension field generated by the set $$\{2\pi i,\zeta(k),\zeta(a,b)\}_{k\in (I(a,b)\cup J(a,b))\setminus\{0,N\}}$$
is maximal (where the sets $I(a,b)$ and $J(a,b)$ are introduced in Definitions \ref{inm},\ref{jnm}).
\end{conj}
\begin{prop}
The period conjecture implies \Cref{corodimtrans}.
\end{prop}
\begin{proof}
The period conjecture predicts that the dimension of $G(a,b)$ (computed in \Cref{dimi}) equals the transcendence degree of the algebra generated by the coefficients of the period matrix of $M(a,b)$ computed in \Cref{permat}. Hence, it predicts that the transcendence degree of the algebra generated by $$\{(2\pi i)^k,\zeta(k),(2\pi i)^{k}\zeta(N-k),\zeta(a,b)\}_{k\in I(a,b)\setminus\{0,N\}}$$ is $2d(a,b)-\vert I(a,b)\cap J(a,b)\vert$. \\

Let us denote by $A$ the following ring:
$$A:=\Q[\{(2\pi i)^k,\zeta(k),(2\pi i)^{k}\zeta(N-k),\zeta(a,b)\}]_{k\in I(a,b)\setminus\{0,N\}}$$ and let $B$ be the ring: $$B:=\Q[\{2\pi i,\zeta(k),\zeta(N-k),\zeta(a,b)\}]_{k\in I(a,b)\setminus\{0,N\}}.$$
Let us show that $$\mathrm{trdeg}(\mathrm{Frac}\,A)=\mathrm{trdeg}(\mathrm{Frac}\,B)$$ where $\mathrm{trdeg}$ is the transcendence degree and $\mathrm{Frac}\,(-)$ denotes the fraction field.\\
We have an obvious surjective morphism of rings $$\Q[p,z_k,z_{a,b}]_{k\in (I(a,b)\cup J(a,b))\setminus\{0,N\}}\twoheadrightarrow \Q[p^k,z_k,p^{k}z_{N-k},z_{a,b}]_{k\in I(a,b)\setminus\{0,N\}}$$ for any variables $p$, $z_k$, $z_{a,b}$. We consider the case when $p$, $z_k$ and $z_{a,b}$ are respectively $2\pi i$,  $\zeta(k)$ and $\zeta(a,b)$. We obtain a finite morphism of schemes $$\Spec(A)\rightarrow \Spec(B)$$ which preserves the dimension. Hence the transcendence degree of the fraction fields is the same. Thus we have $$\mathrm{trdeg}(\mathrm{Frac}\,A)=\mathrm{trdeg}(\mathrm{Frac}\,B)$$ as announced.\\

Since the period conjecture predicts that the transcendence degree of $\mathrm{Frac}\,A$ is $$\mathrm{trdeg}(\mathrm{Frac}\,A)=2d(a,b)-\vert I(a,b)\cap J(a,b)\vert$$ it implies that the transcendence degree of $\mathrm{Frac}\,B$ is the same.\\
Since there are exactly $2d(a,b)-\vert I(a,b)\cap J(a,b)\vert$ elements in the set $$\{2\pi i, \zeta(k),\zeta(a,b)\}_{k\in I(a,b)\setminus\{0,N\}}$$ it means that the transcendence degree is maximal. We obtain \Cref{corodimtrans}.
\end{proof}

\begin{rem}
Let us now suppose that $a$ and $b$ have different parities. Then using \Cref{remz}\ref{twists2}, the value $\zeta(a,b)$ is a sum of twists of single zeta values. Hence applying the period conjecture, we just obtain the conjecture that the transcendence degree of the $\Q$-algebra generated by the set $$\{2\pi i, (2\pi i)^{N-k}\zeta(k)\}_{k\in I(a,b)\setminus\{0\}}$$ is $d(a,b)$. But the period conjecture already predicts that the transcendence degree of the $\Q$-algebra generated by the set $$\{2\pi i, \zeta(k)\}_{k\in I(a,b)\setminus\{0\}}$$ is $d(a,b)$. Hence, we obtain no new conjecture in that case. 
\end{rem}

\printbibliography

IRMA, Université de Strasbourg et CNRS, 7 rue René-Descartes, 67084 Strasbourg, France\\

E-mail address: kenza.memlouk@math.unistra.fr
\end{document}